\documentclass[cjm]{ipart}
\usepackage[T1]{fontenc}
\usepackage{amsfonts}
\usepackage{mathrsfs}
\usepackage{amscd,amsmath,amssymb,amsfonts,amsthm}
\usepackage[cmtip, all]{xy}

\pubyear{2016}
\volume{0}
\issue{0}
\firstpage{1}
\lastpage{999}
\setattribute{copyright}{text}{}

\startlocaldefs
\theoremstyle{plain}
\newtheorem{thm}{Theorem}
\newtheorem{lem}[thm]{Lemma}
\newtheorem{cor}[thm]{Corollary}
\newtheorem{prop}[thm]{Proposition}
\newtheorem{conj}[thm]{Conjecture}
\theoremstyle{definition}

\newtheorem{rmk}[thm]{Remark}
\newtheorem{ex}[thm]{Example}

\newtheorem{claim}[thm]{Claim}

\numberwithin{thm}{section} \numberwithin{equation}{section}

\newcommand{\eq}[2]{\begin{equation}\label{#1}#2 \end{equation}}

\newcommand{\ga}[2]{\begin{gather}\label{#1}#2 \end{gather}}

\newcommand{\Pic}{{\rm Pic}}
\newcommand{\Div}{{\rm div}}
\newcommand{\Hom}{{\rm Hom}}

\newcommand{\Spec}{{\rm Spec \,}}

\newcommand{\sC}{{\mathcal C}}

\newcommand{\sH}{{\mathcal H}}
\newcommand{\sI}{{\mathcal I}}

\newcommand{\sK}{{\mathcal K}}

\newcommand{\sO}{{\mathcal O}}

\newcommand{\sR}{{\mathcal R}}

\newcommand{\C}{{\mathcal  C}}

\newcommand{\Q}{{\mathbb Q}}

\newcommand{\Z}{{\mathbb Z}}

\newcommand{\et}{{\rm \acute{e}t}}

\newcommand{\Nis}{{\rm Nis}}

\newcommand{\CH}{{\rm CH}}
\newcommand{\Zcyc}{{\rm Z}}

\newcommand{\ch}{{\rm ch}}

\newcommand{\sm}{{\rm sm}}
\newcommand{\sing}{{\rm sing}}

\endlocaldefs

\begin{document}
\makeatletter
\safe@footnotetext{\emph{Date}: March 27th, 2015; revised on April 29th, 2016.}
\makeatother

\begin{frontmatter}
\title{A restriction isomorphism for cycles of relative dimension zero}
\runtitle{A restriction isomorphism for cycles of relative dimension zero}

\author{\fnms{Moritz} \snm{Kerz}\thanksref{t1}\ead[label=e1]{moritz.kerz@mathematik.uni-regensburg.de}}
\address{Fakult\"at f\"ur Mathematik, Universit\"at Regensburg,\\93040 Regensburg, Germany\\
\printead{e1}}
\thankstext{t1}{Supported by the DFG
Emmy Noether-Nachwuchsgruppe ``Arithmetik \"uber endlich erzeugten K\"orpern'', the DFG
SFB 1085 and the Institute for Advanced Study.}

\author{\fnms{H{\'e}l{\`e}ne} \snm{Esnault}\thanksref{t2}\ead[label=e2]{esnault@math.fu-berlin.de}}
\thankstext{t2}{Supported by the Einstein Foundation and the ERC Advanced Grant 226257.}
\address{FU Berlin, Mathematik, Arnimallee 3, 14195 Berlin, Germany\\
\printead{e2}}

\author{\fnms{Olivier} \snm{Wittenberg}\corref{}\ead[label=e3]{wittenberg@dma.ens.fr}}
\address{D\'epartement de math\'ematiques et applications,
\'Ecole normale\\ sup\'erieure, 45~rue d'Ulm,
75230 Paris Cedex 05,
France\\
\printead{e3}}

\runauthor{Moritz Kerz, H\'el\`ene Esnault and Olivier Wittenberg}

\begin{abstract} 
We study the restriction map to the closed fiber of a regular projective scheme over an excellent henselian
discrete valuation ring, for a cohomological version
of the Chow group of relative zero-cycles. Our main result extends the work of Saito--Sato
\cite{SS10} to general perfect residue fields.
\end{abstract}

\end{frontmatter}

\section{Introduction} 

Let $A$ be an excellent henselian discrete valuation ring with perfect residue field $k$ of
exponential characteristic $p$.
Let $X/A$ be a regular connected scheme which is projective and flat over $A$, such that the reduced special fiber $Y\subset X$ is
a simple normal crossings divisor. Let $d$ be the dimension of $Y$.

\smallskip

Let $\Lambda$ be a commutative ring.
We denote by  $H^{i,j}(Y):=H^i(Y_{\rm cdh},\Lambda(j))$ the motivic cohomology group defined by
Suslin--Voevodsky as the cdh-cohomology of the motivic complex $\Lambda(j)$ \cite{SV02}. In particular, 
for $Y/k$ smooth, there is an isomorphism $H^{2j,j}(Y)\cong \CH^j(Y)_\Lambda$ onto
the Chow group of codimension $j$ cycles with $\Lambda$-coefficients.
Therefore one can view $H^{2d,d}(Y)$ as a
cohomological version of the Chow group of zero-cycles up to rational equivalence, while  for singular
$Y$,  the usual Chow groups form  a Borel--Moore homology theory.  
Note that for $d=1$ we have $H^{2,1}(Y)=\Pic(Y)\otimes \Lambda$.

\smallskip

We say that an integral one-cycle $Z\subset X$ is \emph{in general position} if $Z$ is flat over $A$
and $Z$ does not meet the singular locus of $Y$. Let $\Zcyc^g_1(X)$ be the free $\Z$-module
generated by these integral one-cycles in general position. An integral one-cycle in
general position $Z\subset X$
can be restricted to a zero-cycle  $[Y\cap Z]\in \Zcyc_0(Y^\sm)$  on the smooth locus
$Y^\sm$ of $Y$. This yields  a homomorphism of abelian groups
\[
\tilde \rho: \Zcyc_1^g(X) \to \Zcyc_0(Y^\sm).
\]   

For a closed point $y\in Y^\sm$, there is a purity isomorphism  $\Lambda \cong H_y^{2d,d}(Y)$ for motivic
cohomology with support in $y$. These isomorphisms induce a homomorphism
of $\Lambda$-modules $\Zcyc_0(Y^\sm)_\Lambda\to H^{2d,d}(Y)$.  On the other hand, denoting by $\CH_1(X)$ the Chow group of one-cycles on $X$, 
one has canonical homomorphisms  $\Zcyc_1^g(X)\hookrightarrow \Zcyc_1(X) \to \CH_1(X)$. 
We can now formulate the main theorem of this note. 

\begin{thm}\label{thm.main} 
Let $\Lambda=\Z/m\Z$ with $m$ prime to $p$.
\begin{itemize}
\item[1)]  If $k$ is finite or algebraically closed, or $(d-1)!$ is prime to $m$, or $A$
  has equal characteristic, or $X/A$ is smooth, then there is a unique $\Lambda$-module homomorphism $\rho$ making the diagram
\[
\xymatrix{
\Zcyc_1^g(X)_\Lambda \ar[r]^{\tilde \rho} \ar[d] & \Zcyc_0(Y^\sm)_\Lambda \ar[d] \\
\CH_1(X)_\Lambda \ar[r]_-{\rho}  & H^{2d,d}(Y)
}
\]
commutative. 
\item[2)] If $\rho:\CH_1(X)_\Lambda \to H^{2d,d}(Y)$ is a homomorphism of $\Lambda$-modules making the above diagram commutative, then it is an isomorphism.
\end{itemize}
\end{thm}

Up to an explicit presentation of $ H^{2d,d}(Y)$, established in Theorem~\ref{thm.motcomp} below,
part 1) is discussed in Sections~\ref{sec:rest} and~\ref{sec:proof} and part 2) is  Corollary~\ref{cor.isorho}.

\smallskip

In the case when $k$ is finite or algebraically closed, Theorem~\ref{thm.main} is equivalent to and reproves 
results of Saito--Sato \cite{SS10} in view of the \'etale
realization isomorphism, Proposition~\ref{prop.compet} below. In fact, in some sense, our arguments extend
to general residue fields $k$
the simplified approach to the results of Saito and Sato which is developed in \cite[App.]{EW13}.

\smallskip

By the moving lemma of Gabber--Liu--Lorenzini \cite{GLL13},  $\CH_1(X)$ is
generated by $\Zcyc_1^g(X)$, so uniqueness in Theorem~\ref{thm.main} is clear.  Surprisingly,
one of the main difficulties is the construction of $\rho$, which, as of today, can be performed only under
the extra assumptions of 1).  It is expected that the homomorphism $\rho$ with the
properties of part 1) exists in general for any ring of
coefficients $\Lambda$. 
So far, we do not have the necessary motivic techniques at our disposal to show this. 

\smallskip

The proof of part 1), in Section~\ref{sec:proof},  rests,
in equal characteristic, on the Gersten
conjecture for Milnor $K$-theory  due to the first author
\cite{Ker09};
when the residue field is finite,
on the \'etale realization  using the
Kato conjecture established in \cite{KS12};
when $(d-1)!$ is
prime to $m$, on algebraic $K$-theory and the Grothendieck--Riemann--Roch theorem \cite{BS98}.

\smallskip

As for 2), the proof  in Sections~\ref{sec:lift} and~\ref{sec:inverse} relies on an explicit geometric
presentation of $H^{2d,d}(Y)$, which is discussed in Sections~\ref{sec.cg1}  and~\ref{sec.motive}, together with
an idea of  Bloch \cite[App.]{EW13} which enables one to construct an inverse to $\rho$ in
Section~\ref{sec:inverse}. Note that for 2), it is essential that $\Lambda=\Z/m\Z$ and that  $m$ be prime to $p$.

\smallskip
When $X/A$ is smooth,
the restriction map $\rho$ factors as
\[
\CH_1(X) \twoheadrightarrow \CH_0(X_K) \xrightarrow{sp} \CH_0(Y)
\]
where~$sp$ denotes the specialization homomorphism \cite[20.3.1]{Ful84}.
Here we use the identification $H^{2d,d}(Y)=\CH_0(Y)_\Lambda$ for $Y/k$ smooth and $
\Lambda$ arbitrary \cite[Thm.\ 19.1]{MVW}.
So one deduces from Theorem~\ref{thm.main}
\begin{cor} \label{cor.main} Let $\Lambda=\Z/m\Z$ with $m$ prime to $p$. If $X/A$ is smooth, the 
specialization homomorphism $sp: \CH_0(X_K)_\Lambda \to \CH_0(Y)_\Lambda$ is an isomorphism.
\end{cor}

As an application of Theorem~\ref{thm.main}, we show, in Proposition~\ref{prop.laur}, that for a formal Laurent power series field $K=k((\pi))$
with $[k:\Q_p]<\infty$, the Chow group $\CH_0(X_K)_\Lambda$ is finite for
$\Lambda=\Z/m\Z$ assuming $p$ is prime
to $m$.

In the final section we state a conjecture describing $\CH_1(X)_\Lambda$ for the case $\Lambda=\Z/m\Z$ with $m$ not
necessarily prime to $p$. We also explain in which motivic bidegrees we expect the analogue
of Theorem~\ref{thm.main} to be true.

\section{The cycle group $\C(Y)$  of a simple normal crossings variety over a perfect field}  \label{sec.cg1}

\subsection{ Simple normal crossings varieties} \label{ss:snc}

In this section we make precise the notion of simple normal crossings varieties we shall use throughout the article. 

\smallskip 

Let $k$ be a perfect field of exponential characteristic $p$ with algebraic closure $\bar k$. 
 A  $d$-dimensional $k$-scheme $Y$ is a {\it normal crossings variety} (nc
 variety)  if it is reduced, quasi-projective (which implies of finite type) and if
for any closed point  $y\in Y_{\bar k}$, 
the henselized local ring 
$ \sO_{Y_{\bar k},y}^h$ is $\bar k$-isomorphic  to $\bar k [y_1,\ldots, y_{d+1}]^h /( \prod_{i=1}^r
 y_i) $ ($1\le r\le d+1$).
The variety $Y$ is a {\it simple normal crossings variety} (snc variety) if it is a nc variety and additionally every
irreducible component of $Y$ is smooth over $k$.
Here the henselization $R[y_1,\dots ,y_r]^h$ of a polynomial ring over a ring $R$ means
henselization in the ideal $(y_1,\ldots , y_r)$.

By a {\it simple normal crossings divisor} (snc divisor) we denote what is called a strict normal crossings
divisor in \cite[2.4]{dJ96}. We use the notation of {\it normal crossings divisor} (nc
divisor) as in {\it loc.\ cit}.  We recall that if~$Y$ is a nc divisor on a scheme~$X$,
then by definition~$X$ is regular at the points of~$Y$.

\smallskip 

Let $A$ be a henselian discrete valuation ring with residue field $k$.  Let $\bar A$ be the
integral unramified extension of $A$ with residue field $\bar k$. For a scheme $X/A$ we
write $X_{\bar A}$ for the scheme $X\times_A \bar A$.
We denote by $\pi$ a fixed prime element of $A$ (thus of $\bar A$).
\begin{lem} \label{prop:sncv}
Let $X$ be a scheme which is flat and of finite type over $A$. If the reduced special
fiber $Y$ is a nc divisor,  the henselized local ring
$\sO_{X_{\bar A},y}^h$ at a closed point $y$ of $ Y_{\bar k}$  is $\bar A$-isomorphic to  \eq{eq.locpre}{\bar A[y_1,\ldots,
y_{d+1}]^h /(\prod_{i=1}^r y_i ^{m_i} -\pi u) , } where  $u$ is a unit of
$\bar A [y_1,\ldots, y_{d+1}]^h$ and the $m_i$ are positive integers ($r\ge 1$).
\end{lem}
\begin{proof}
Choose generators $y_1,\ldots , y_{d+1}$ of the maximal ideal of $ \sO_{X_{\bar
    A},y}^h$ such that the product of the first $r$ defines the snc divisor $Y_{\bar k}$
on $X_{\bar A}$ around $y$. The divisor of $\pi$ on the regular scheme $\Spec \sO_{X_{\bar
    A},y}^h$ is of the form $\Div (y_1^{m_1}\cdots y_r^{m_r})$ for certain $m_i$.
So the canonical surjection 
\[
 \bar A [y_1,\ldots, y_{d+1}]^h \to  \sO_{X_{\bar  A},y}^h
\]
factors through a ring of the form \eqref{eq.locpre} with a unit $u$, since $\bar A [y_1,\ldots,
y_{d+1}]^h$ is regular. By dimension reasons this map is an
isomorphism as the ring \eqref{eq.locpre} is integral.
\end{proof}

\begin{lem}\label{lem.divvar}
Let $X$ be a scheme which is flat and of finite type over $A$.
The following conditions are equivalent:
\begin{itemize}
\item[(i)] the reduced special fiber $Y/k$ of $X$ is a
nc variety and $X$ is regular at the points of $Y$,
\item[(ii)]  $Y\subset X$ is a nc divisor (thus $X$ is regular at the points of Y). 
\end{itemize}
This assertion remains true if one replaces nc with snc.
\end{lem}

\begin{proof}
The only nontrivial part is to show that if $Y$ is a nc divisor on $X$ it is
a nc variety over $k$. But this follows from Lemma~\ref{prop:sncv}.
\end{proof}

A $1$-dimensional closed subscheme $C$ of a snc variety~$Y$ is a {\it simple normal crossings
  subcurve}  (snc subcurve) if  its (scheme-theoretic) intersections with the irreducible components $Y_i$ of $Y$ are smooth and purely $1$-dimensional, if its intersection with $Y_i\cap Y_j$ is reduced and purely $0$-dimensional
for all $i\neq j$ and if its intersection with $Y_i\cap Y_j\cap Y_\ell$ is empty
for all $i\neq j\neq \ell\neq i$.
We stress that according to this
definition, if a smooth irreducible curve contained in~$Y_i$ is
a snc subcurve of~$Y$, then it must be disjoint from~$Y_i \cap Y_j$ for
every $j \neq i$.

\begin{lem} \label{prop:sncc}
A snc subcurve $C$ of a snc variety $Y$ is regularly immersed and is itself a one-dimensional snc variety.
\end{lem}
\begin{proof}
Consider  a closed point $y$ of $Y_{\bar k}$ lying on $C$. For $y\in Y^\sm$ the curve $C$
is regular at $y$. For $y\in Y_1\cap Y_2$
we choose an isomorphism
\[\widehat \sO_{Y_{\bar k},y} \cong \bar k [[y_1,\ldots, y_{d+1}]]/(y_1 y_2)=: \widehat
\sO.   \]
Let $\widehat \sI$ be the ideal defining $C$ in $\widehat \sO$. As $C\cap Y_1$ and $Y_1$
are regular the ideal $\widehat \sI \cdot \widehat \sO/(y_2)$ is generated by a regular
sequence $z'_3, \ldots , z'_{d+1}$. Lift this sequence to $z_3, \ldots , z_{d+1}\in
\widehat \sI $. Then \[\widehat \sO\cong \bar k [[y_1,y_2,z_3,\ldots, z_{d+1}]]/(y_1
y_2)  \]
and $C$ is defined by the regular sequence $z_3,\ldots ,z_{d+1}$.
\end{proof}

\subsection{Cycle group of a simple normal crossings variety} \label{ss:C(Y)}

Throughout the article, we let $\Lambda$ be a commutative ring  with unity.  All
 motivic cohomology groups  have $\Lambda$-coefficients, while Chow groups have integral coefficients,
unless otherwise specified.

\smallskip

Let us assume from now on that $Y/k$ is a snc variety of dimension $d$.
The abelian group $\C(Y)$ we {define in this section} can be thought of as a cohomological variant of the Chow
group of zero-cycles of $Y$ with $\Lambda$-coefficients. Recall that, in contrast,  the usual Chow groups
form a Borel--Moore homology
theory. In fact, we will see {in Section~\ref{sec.motive}}
that for $p$ invertible in $\Lambda$ and $Y$ proper,  the group $\C(Y)$ is
isomorphic to motivic cohomology $H^{2d}(Y_{\rm cdh},\Lambda(d))$ in the sense of Suslin--Voevodsky
\cite{SV02}.

\smallskip

Let $Y^\sm$ be the smooth locus of $Y$ over $k$ and $Y^\sing$ its complement with the
reduced subscheme structure.
  Let $\Zcyc_0(Y^\sm)$ be the free $\Z$-module generated by the integral
zero-dimensional subschemes of $Y^\sm$.  The group $\C(Y)$ is defined by an exact
sequence 
$$0\to  {\sR } \to \Zcyc_0(Y^\sm)_\Lambda \to \C(Y) \to 0$$ 
where $\sR$ is the sub-$\Lambda$-module  of $\Zcyc_0(Y^\sm)_\Lambda$
generated by divisors ${\rm div}(g)$ associated with rational
functions $g$ on
curves $C$ of the following two types.

\medskip

\noindent
{\it Type 1 data:} A type 1 datum is a pair $(C_1,g_1)$. Here $C_1$ is an integral
one-dimensional closed subscheme $C_1\subset Y$ which is not contained in $Y^\sing$. Let
$\eta$ be the generic point of $C_1$ and let $C_1^\infty$ be $\tilde C_1\times_Y Y^\sing$
with the reduced subscheme structure, where $\tilde C_1$ is the normalization of $C_1$. 
The rational function $g_1$ is an element of
$\ker(\sO_{C_1,C_1^\infty\cup \{\eta\}}^\times \to \sO_{C_1^\infty}^\times)$, that is, it is a  rational function on $C_1$, which is defined and 
equal to $1$ over $C_1^\infty$.

\medskip

\noindent
{\it Type 2 data:} A type 2 datum is a pair $(C_2,g_2)$. Here $C_2\subset Y$ is a snc subcurve on $Y$. 
Let $C_2^\infty$  be the finite union of $C_2\cap Y^\sing$ and the set of maximal points of $C_2$.
The function $g_2$ is an element of $\sO_{C_2,C_2^\infty}^\times$.

\medskip

The group $\C(Y)$ is a variant of the cycle group studied in
\cite{LW85}. Note that if $Y/k$ is smooth, then $\C(Y)= \CH_0(Y)_\Lambda$.

\section{The restriction homomorphism} \label{sec:rest}

Let $A$ be an excellent henselian discrete valuation ring with perfect residue field $k$ of exponential characteristic $p$. Let $X$ be a
regular connected scheme which is projective and flat over $A$, such that the reduced
special fiber $Y$ of $X$ is a snc divisor. So $Y/k$ is a snc variety by Lemma~\ref{lem.divvar}. Let $\Lambda$ be a commutative ring with unity.

\smallskip

By $\Zcyc_1(X)$ we denote the free $\Z$-module generated by the integral closed
subschemes $Z\subset X$ of dimension one. We write $[Z]\subset \Zcyc_1(X)$ for the associated
generator. The  subgroup $\Zcyc_1^g(X)\subset \Zcyc_1(X)$ of `generic' elements is by definition
generated by those integral cycles $[Z]$ such that $Z$  is flat over $A$ and $Z\cap Y^\sing =
\varnothing$.

\smallskip

To a zero-dimensional closed subscheme $S\subset Y^\sm$,  one 
associates as usual  the element $$[S]=\sum_{z\in S}{\rm length} (\sO_{S,z}) [z]\in \C(Y).$$ Here length is the length of a ring as a module over itself.

\smallskip

One obtains the pre-restriction homomorphism 
\eq{tilderho}{\tilde \rho: \Zcyc_1^g(X) \to \Zcyc_0(Y^\sm),\quad  [Z] \mapsto
[Z\cap Y]. } 
In Section~\ref{sec:proof} we prove the following theorem.

\begin{thm}[Restriction]\label{thm.res} Assume that $\Lambda=\Z/m\Z$ with $m$ prime to $p$  and that
  additionally one of the
  following conditions holds:
\begin{itemize}
\item[(i)] $k$ is algebraically closed,
\item[(ii)] $k$ is finite,
\item[(iii)] $(d-1)!$ is prime to $m$, 
\item[(iv)] $A$ is equicharacteristic,
\item[(v)] $X/A$ is smooth.
\end{itemize}
 Then there 
 is a unique restriction homomorphism $\rho: \CH_1(X)_\Lambda \to \C(Y)$ such that the diagram
\[
\xymatrix{
\CH_1(X)_\Lambda \ar[r]^-{\rho} & \C(Y) \\
 \Zcyc_1^g(X)_\Lambda  \ar[u] \ar[r]_{\tilde \rho} & \Zcyc_0(Y^\sm)_\Lambda \ar[u]
}
\]
is commutative.
\end{thm}

\begin{rmk}\label{rmk.reiii}
Rather than proving Theorem \ref{thm.res} under the assumption (iii),
we will see, more generally, that the restriction homomorphism exists
for any ring $\Lambda$ in which $p (d-1)!$ is invertible.
\end{rmk}

By \cite[Thm.\ 2.3]{GLL13}, the
left vertical map in the diagram is surjective for any ring $\Lambda$,  so uniqueness of $\rho$
is clear.
A direct approach to the construction of $\rho$ would be to prove a moving lemma. Unfortunately, such Chow
type moving lemmas are not available in mixed characteristic. We explain in Section~\ref{sec:proof} how to use the extra hypotheses to yield an indirect proof. 

A general theory of derived mixed motives over base schemes, which satisfies standard properties, would show that the
restriction map exists for any $\Lambda$ in complete generality. Unfortunately, such a
theory is not available at the moment.

\section{Canonical lifting of zero-cycles} \label{sec:lift}

We use the notation of Section~\ref{sec:rest}, in particular recall we considered the solid
arrows in the diagram 
\eq{lift.dia}{ \begin{aligned} \xymatrix{
    \CH_1(X)_\Lambda  &  \\
    \Zcyc_1^g(X) \ar[u] \ar[r]_{\tilde \rho} & \Zcyc_0(Y^\sm)\rlap{\text{.}} \ar@{-->}[lu]_{\tilde
      \gamma} } \end{aligned} } 
It is well known that the map $\tilde \rho$ is surjective, i.e.\
one can lift zero-cycles on $Y^\sm$ to flat one-cycles on $X/A$. 
As the idea for this construction is central for the arguments in this section and the
next one, we  recall it. Given a closed point $y\in Y^\sm$,  let $(a_1,\ldots ,a_d)\in
\sO_{Y^\sm,y}$ be a regular sequence generating the maximal ideal. Lift these parameters
to elements $\hat a_1 , \ldots , \hat a_d\in \sO_{X,y}$. Let $Z^{\rm loc}\subset \Spec
\sO_{X,y}$ be the associated closed subscheme of the ideal \[\hat a_1  \sO_{X,y} +  \cdots
+ \hat a_d  \sO_{X,y}\subset \sO_{X,y}\]
and let $Z$ be the unique irreducible component of the closure of $Z^{\rm loc}$ in $X$
which contains $y$. Then $Z$ is finite, flat over $A$ and $Z\cap Y$ is the integral
scheme associated with the point $y\in Y$, so $Z$ defines a lifting of $y$ to  $ \Zcyc_1^g(X)$.

\smallskip

Of course this lifting is
not unique if  $d=\dim(Y)>0$. However under certain conditions the lift is
unique in the Chow group, i.e.\ there exists a unique dashed map $\tilde \gamma$ making
the diagram \eqref{lift.dia} commutative. The goal of this section is to explain this fact.
We use the notation of Theorem~\ref{thm.res}.

\begin{prop}\label{key.prop}
 Let $X$ be a regular scheme, flat and projective over $A$,
whose reduced special fiber $Y$ is a snc divisor on $X$.
Let
  $\Lambda=\Z/m\Z$ with $m$ prime to~$p$.
The group
$\ker\big(\tilde \rho: \Zcyc_1^g(X)_\Lambda \to \Zcyc_0(Y^\sm)_\Lambda \big)$
is contained in $\ker \big( \Zcyc_1^g(X)_\Lambda \to
\CH_1(X)_\Lambda \big)$.  In particular, there exists a unique
homomorphism $\tilde \gamma:\Zcyc_0(Y^\sm) \to \CH_1(X)_\Lambda$ making the diagram \eqref{lift.dia} commutative.
\end{prop}

\begin{proof}
We first assume $d=1$.
There is a commutative diagram involving \'etale cycle maps
\[
\xymatrix{
\Zcyc^g_1(X)_\Lambda \ar[r] \ar[d]_{\tilde \rho} & \CH_1(X)_\Lambda  \ar[r]^-{c_X} & H^2(X_{\et},\Lambda(1)) \ar[d]^\wr \\
\Zcyc_0(Y^\sm)_\Lambda  \ar[rr] &  & H^2(Y_{\et},\Lambda(1))\rlap{\text{.}}
}
\]
The right vertical isomorphism is due to proper base change for \'etale cohomology \cite[IV,~Thm.~1.2]{SGA4.5}. The cycle map $c_X$ is
injective as it is equal to the composite map
\ga{}{\CH_1(X)_\Lambda\cong \Pic(X)_\Lambda
 \hookrightarrow  H^2(X_{\et}, \Lambda(1))
\notag}
which is injective 
by Kummer theory. 
So the proposition follows in this case.

\smallskip

For $d$ general fix a projective
embedding $X\subset \mathbb P^N_A$.   If $k$ is finite, the assertion to be proved is
invariant under extension of $k$ to its maximal pro\nobreakdash-$\ell$-extension, for some prime
$\ell$ prime to $m$. For this, one applies the flat pull-back
and proper push-forward  \cite[p.~394]{Ful84}.  So we may assume $k$ to be infinite. 

\smallskip

Consider an element
\[
\sum_{i=1}^s r_i [Z_i] \in \ker (\tilde \rho)
\]
with $Z_i$ an integral scheme, flat over $A$,  disjoint from $Y^\sing$.
Then $(Z_i\cap Y)_{\rm red}$  consists of a closed point  $z_i$
 of $Y^\sm$. 
 Choose a lift $Z'_i$ of $z_i$,
i.e.\ an integral one-cycle $Z'_i\subset X$ which is finite, flat over $A$, and such that
$Z'_i\cap Y=z_i$ in the sense of schemes. When choosing the lifts $Z'_i$ we can make sure that
$Z'_i=Z'_j$
whenever
$z_i=z_j$.
This condition implies that $\sum_{i=1}^s r_in_i [Z'_i] =0$ in $\Zcyc_1^g(X)_\Lambda$, where
 $n_i$ is the intersection multiplicity of $Z_i$ and $Y$. 
Hence, in order to prove the proposition,
it is enough to show that  $$\sigma( [Z_i]-n_i [Z'_i])=0$$
for all $1\le i\le s$. Here $\sigma: \Zcyc_1^g(X)_\Lambda
\to \CH_1(X)_\Lambda $ is the canonical map.
So we need only consider an element of the form $[Z]-n[Z']\in \ker(\tilde \rho)$, i.e.\ $Z', \ Z$ are {one-dimensional integral,}
finite, flat {schemes}
over $A$, with $Z'\cap Y$ reduced (in particular $Z'$ is regular) and $n$ is the intersection multiplicity of $Z$ and $Y$.

\smallskip

We first assume that $Z$ is regular. 
We prove this special case by induction on $d>1$, using a Bertini type argument.
Later we explain how to reduce to this special case.
As by assumption, $H^0(Z,\sO_Z)$ is a discrete valuation ring, the embedding dimension of
$Z\cap Y$ is {at most} one. So by Bertini's theorem  \cite{AK79} there is an ample hypersurface section $H_Y$ of
$Y$ such that $Z\cap Y $ is contained in $H_Y$ and such that $H_Y$ is a simple normal
crossings subvariety of~$Y$. If $H_Y$ is chosen ample enough we can lift it to a
hypersurface section $H$ of $X$ flat over $A$ and containing $Z$, see \cite[Thm.\ 1]{JS12}
and \cite[Sec.\ 4]{SS10}. Then $H$ is 
regular and its reduced special fiber is a snc divisor.

\smallskip

Next we use Bertini's theorem  ({\it loc.\ cit.}) to find a snc subvariety $H'_Y$ of $Y$ which satisfies
\begin{itemize}
 \item $H'_Y$ is of dimension one,
\item  $Y\cap Z'$ is contained in $H'_Y$,
\item $H'_Y$ is the intersection of $d-1$ ample hypersurface sections,
\item $H'_Y\cap H_Y$ consists of reduced points in $Y^\sm$. 
\end{itemize} 

{ If the hypersurface sections are chosen ample enough, we can lift $H'_Y$ to a closed 
subscheme $H'\subset X$, regular, flat over $A$, such that $(H'\otimes k)_{\rm red}$ is a
snc divisor on $H'$ and  such that $H'$ contains $Z'$.}
The scheme $H\cap H'$ is finite, flat
over $A$ and its intersection with $Y$ is reduced. In particular $H\cap H'$ is regular. Let $Z''$ be its component containing $Z'\cap Y$.

\smallskip 

By the induction assumption, $[Z]-n[Z'']$ vanishes in $\CH_1(H)_\Lambda$,  thus it also
vanishes in $\CH_1(X)_\Lambda$, so $\sigma([Z]-n[Z''])=0$. By the $d=1$ case of the
proposition, $[Z']-[Z'']$ vanishes in $\CH_1(H')_\Lambda$, thus it also vanishes  in $\CH_1(X)_\Lambda$, 
so $\sigma ([Z']-[Z''])=0$. Finally, we obtain
\[
\sigma ([Z]-n [Z'])=  \sigma([Z]-n[Z'']) +n \sigma ([Z'']-[Z'])= 0.
\]

To finish the proof we apply Bloch's idea in \cite[App.]{EW13} to  treat the general case.
Let $\bar Z$ be
the normalization of $Z$. As $A$ is excellent, $\bar Z$ is finite over $Z$,  we can find
a projective embedding $\bar Z\to \mathbb P^M_X$. Let $\bar X$ be $\mathbb P^M_X$ and
$\bar Y$ be~$\mathbb P^M_Y$.  Choose a lift of the closed point $(\bar Z\cap \bar Y)_{\rm
  red}$ to a closed subscheme  $\bar Z' \subset \bar X$  which
is \'etale over {$Z'$. } Let $\bar n$ be the intersection multiplicity of $\bar Z$ and $\bar
Y$. Then the push-forward of $[\bar Z] - \bar n [\bar Z']$ along $\bar X \to X$ is equal to
our old cycle $[Z]-n [Z']$. 

The cycle $[\bar Z] - \bar n [\bar Z']$ is in the kernel of the
restriction map
\[
\bar{\tilde \rho}: \Zcyc^g_1(\bar X)_\Lambda \to \Zcyc_0(\bar Y^\sm)_\Lambda,
\] 
so it vanishes in $\CH_1(\bar X)_\Lambda$ by what is shown above in the case of regular cycles. Finally, this
implies that $\sigma ([Z]-n [Z'])$ vanishes too.
\end{proof}

\section{The inverse restriction map $\gamma$}   \label{sec:inverse}

Recall that in Section~\ref{sec:lift} we 
constructed a canonical homomorphism
\[
\tilde \gamma: \Zcyc_0(Y^\sm)_\Lambda\to \CH_1(X)_\Lambda.
\]

\begin{thm}
{Under the assumptions of Proposition~\ref{key.prop}}, 
there is a unique homomorphism $\gamma: \C(Y) \to \CH_1(X)_\Lambda$ making the diagram
\[
\xymatrix{
 \C(Y) \ar[r]^-{\gamma} & \CH_1(X)_\Lambda \\
 \Zcyc_0(Y^\sm)_\Lambda \ar[u] \ar[ur]^{\tilde \gamma}  &   \Zcyc_1^g(X)_\Lambda \ar[l]^-{\tilde \rho}  \ar[u]_{\sigma}
}
\]
commutative. In other words, the map $\tilde \gamma$ factors through $\C(Y)$.
\end{thm}

\begin{proof}
Uniqueness in the theorem is clear,  as by construction, the left vertical arrow is
surjective. In order to prove that the factorization $\gamma$ of $\tilde \gamma$ exists, 
we have to show
that for a type $i$ datum $(C_i,g_i)$ ($i\in \{1,2\}$) on $Y$ we have
$$0= \tilde
\gamma({\rm div}(g_i))  \in \CH_1(X)_\Lambda.
$$

\smallskip
\noindent
{\it Type 2 datum:} \\
Let $(C_2,g_2)$ be a type 2 datum.
The idea is first to lift the curve $C_2$ to a `nice' flat curve $\hat C_2\subset X$ over $A$.
In a second step we lift
$g_2$  to $\hat g_2\in k(\hat C_2)^\times$ such that ${\rm div}(\hat g_2)\in \Zcyc_1^g(X)$ has
restriction ${\rm div}(g_2)\in \Zcyc_0(Y^\sm)$. Then the class of  ${\rm div}(\hat g_2)$ in
$\CH_1(X)_\Lambda$ is $\tilde \gamma({\rm div}(g_2))$ and we are done.

We start with the construction of the lifted curve $\hat C_2$. The argument is similar to
\cite[Lem.\ 2.5]{GLL13}. Let $\mathcal I$ be the ideal sheaf defining $C_2$ in $Y$. By
Lemma~\ref{prop:sncc}, the $\sO_{Y,C_2^\infty}/\sI_{C_2^\infty} $-module
$ \sI_{C_2^\infty} /\sI_{C_2^\infty}^2 $ has a basis $a_1,\ldots , a_{d-1}$. Lift $a_1$ to
an element $\hat a_1\in \sO_{X,C_2^\infty}$ which is a unit at the maximal points of
$Y^\sing$. Let $L(\hat a_1)\subset X$ be the Zariski closure of the closed subscheme of
$\Spec \sO_{X,C_2^\infty} $ given by $\hat a_1=0$. Let $P(\hat a_1)$ be the set of maximal points of
$L(\hat a_1)\cap Y^\sing$. Clearly, the points $P(\hat a_1)$ have height one in $Y^\sing$.
Lift $a_2$ to an element $\hat a_2\in \sO_{X,C_2^\infty}$ which is a unit at the points
$P(\hat a_1)$. Let $L(\hat a_1,\hat a_2)$ be the Zariski closure of the closed subscheme of
$\Spec \sO_{X,C_2^\infty} $ given by $\hat a_1=\hat a_2=0$. Let $P(\hat a_1,\hat a_2)$ be the maximal points
of $L(\hat a_1,\hat a_2)\cap Y^\sing$. Clearly, the points $P(\hat a_1,\hat a_2)$ have height two in
$Y^\sing$. We proceed like this.

In the end we get elements $\hat a_1,\ldots ,\hat a_{d-1}\in \sO_{X,C_2^\infty} $.
Let $\hat C_2^{\rm loc}$ be the closed subscheme of $\Spec \sO_{X,C_2^\infty}$ defined
by the ideal $$ \hat a_1 \sO_{X,C_2^\infty}  + \cdots + \hat a_{d-1} \sO_{X,C_2^\infty}
\subset  \sO_{X,C_2^\infty}$$ and let $\hat C_2$ be the closure of $\hat C_2^{\rm loc}$ in
$X$. 
By Nakayama's lemma the closed subschemes~$\hat C_2 \cap Y$ and~$C_2$ of~$Y$ coincide
in a neighborhood of $C_2^\infty$.
The scheme $\hat C_2$ is regular around the points $C_2^\infty$, because at each
$y\in C_2^\infty$ the intersection of $\hat C_2$ with an irreducible component $Y_1$ of $Y$ is
regular and proper. This implies that $\hat C_2$ is flat over
$A$.  Furthermore, $\hat C_2 \cap Y^\sing $ is finite, because it has dimension zero or is
empty.

Let $\hat C_2^\infty$ be the finite set of points of $\hat C_2$ consisting of the maximal points
 of $\hat C_2\cap Y$ and the   closed  points $\hat C_2\cap Y^\sing$.
In a neighborhood of $\hat C_2^\infty$, the scheme $\hat C_2\cap Y$ is the disjoint union of $C_2$ and a residual
part. Thus, there is a direct product
decomposition of rings
\[
\sO_{\hat C_2\cap Y, \hat C_2^\infty}= \sO_{C_2,C_2^\infty} \times R
\]
where $R$ is a semi-local ring of dimension one. It gives rise to an element $(g_2,1)\in
\sO_{\hat C_2\cap Y, \hat C_2^\infty}^\times$ which we lift to an element $\hat g_2 \in
\sO_{\hat C_2, \hat C_2^\infty}^\times$. 

 As  $\hat C_2^\infty$ contains the points of $\hat C_2\cap Y^\sing$, 
${\rm div}(\hat g_2)$ is in $\Zcyc^g_1(X)$. As restriction of rational functions commutes with
taking associated divisors, we have $\tilde \rho({\rm div}(\hat g_2))={\rm
  div}(g_2)$.  So one has the commutative diagram
 \ga{}{\xymatrix{ & \CH_1(X)_\Lambda & & 0 \\
 \ar[ur]^{\tilde{\gamma}}  \Zcyc_0(Y^{\rm sm}) & \ar[l]^{\tilde{\rho}} \Zcyc_1^g(X) \ar[u] &
 \ar@{|->}[ur]^{\tilde{\gamma}}  {\rm div}(g_2) & \ar@{|->}[l]^{\tilde{\rho}} {\rm div}(\hat{g}_2)\rlap{\text{.}} \ar@{|->}[u]
 } 
\notag }
  This finishes the proof in this case.
  
\smallskip
\noindent
{\it Type 1 datum:} \\
Let $(C_1,g_1)$ be a type 1 datum.
We consider a morphism $\pi:\check X\to X$ which is a repeated blow-up of closed points of
$X$, all closed points lying over $Y^\sing$. We write $\check Y$ for the reduced special
fiber of $\check X$, etc. Note that $\check Y\subset \check X$ is a snc divisor.  By a
careful choice of the closed points,
i.e.\ the centers of the
blow-ups, we can assume that the strict transform $\check C_1$ of $C_1$ in $\check X$
satisfies the following properties (see Jannsen's appendix to \cite{SS10}):

\begin{itemize}
\item[(i)]
$\check C_1$ is smooth around $\pi^{-1}(Y^\sing)$,
\item[(ii)] $\check C_1\cap \check Y^\sing$ consists of reduced points,
\item[(iii)] all points $y\in   \check C_1$   lie on at most two components of   $\check Y$.
\end{itemize}

Let $\tilde{\check\rho}:\Zcyc_1^g(\check X) \to \Zcyc_0(\check Y^\sm)$
and $\tilde{\check\gamma} : \Zcyc_0(\check Y^\sm)\to \CH_1(\check X)_\Lambda$
respectively denote the restriction map
and the map given by Proposition~\ref{key.prop} applied to~$\check X$.
There is a commutative diagram
\begin{align*}
\xymatrix{
\Zcyc_0(\check Y^\sm)\ar@/^2.0pc/[rr]^{\tilde{\check\gamma}} & \Zcyc_1^g(\check X) \ar[l]_(.45){\tilde{\check \rho}} \ar[r] & \CH_1(\check X)_\Lambda \ar[d]^{\pi_*}\\
\Zcyc_0(Y^\sm)\ar@/_2.0pc/[rr]_{\tilde{\gamma}} \ar[u]^{\pi_Y^*} & \Zcyc_1^g(X) \ar[l]_(.45){\tilde\rho} \ar[u] \ar[r] & \CH_1(X)_\Lambda
}
\end{align*}
whose upward vertical arrows are the naive pull-back homomorphisms
and whose downward vertical arrow is the canonical
 push-forward homomorphism
 $\pi_*:\CH_1(\check X) \to \CH_1(X)$
(see \cite[Sec.\ 20.1]{Ful84}).

We are going to show that
$\tilde{\check\gamma}(\pi_Y^*({\rm div}(g_1)))$ vanishes,
which by the above commutative diagram
will imply that $\tilde \gamma ({\rm div}(g_1))$ vanishes.

Let $\check Y_1$ be the irreducible component of $\check Y$ containing $\check C_1$ and
let $\eta$ be the generic point of $\check C_1$. From here we proceed similarly to type 2
data. Choose a basis $a_1, \ldots, a_{d-1}$ of the   $ \sO_{\check Y_1,\check C_1^\infty\cup\{
\eta\}}/\sI_{\check C_1^\infty\cup\{
\eta\}}$-module $\sI_{\check C_1^\infty\cup\{
\eta\}}/\sI^2_{\check C_1^\infty\cup\{
\eta\}}$, where $\sI$ is the ideal sheaf defining $\check C_1$ in $\check Y_1$. 

Exactly as for type 2 data above we lift the $a_i$ to elements 
\[
\hat  a_1, \ldots,\hat a_{d-1}\in \sO_{\check X,\check C_1^\infty\cup\{
\eta\}}
\]
and obtain the closed subscheme $\hat C_1\subset \check X$ as the Zariski closure of the
closed subscheme of $\Spec  \sO_{\check X,\check C_1^\infty\cup\{ \eta\}}$ defined by
$\hat a_1= \cdots = \hat a_{d-1}=0$.  
Note that $\hat C_1$ is regular around $\check
C_1^\infty\cup\{ \eta\}$ and flat over $A$. Choosing the local parameters
as for type~2 data, we can assume that $\hat C_1\cap \check Y^\sing$ consists of only finitely many
points.

Let $\hat C_1^\infty$ be the finite set of points consisting of the maximal points of
$\hat C_1\cap \check Y$ and the points $\hat C_1\cap \check Y^\sing$. 
Let $\check g_1$ be the rational function on $\check C_1 $ induced by $g_1$.  
There is a unique element $\overline
g_1 \in \sO_{\hat C_1\cap Y, \hat C_1^\infty}^\times$ which restricts to $\check g_1$ on $\check C_1 $
and to $1$ on the other irreducible components of  $\hat C_1\cap Y$. We lift $\overline
g_1$ to an element $\hat g_1\in \sO_{\hat C_1,\hat C_1^\infty}$.

We observe that ${\rm div}(\hat g_1)$ is in $\Zcyc_1^g(\check X)$
and that this
element restricts, via
$\tilde{\check\rho}$, to ${\rm div}(\check g_1)\in \Zcyc_0(\check Y^\sm)$. So ${\rm div}(\hat g_1)
=\tilde{\check\gamma} ( {\rm div} (\check g_1) )
=\tilde{\check\gamma} ( \pi_Y^*({\rm div} (g_1)) )
\in \CH_1(\check X)_\Lambda$.  
Hence
$\tilde{\check\gamma} ( \pi_Y^*({\rm div} (g_1)) )$ vanishes in $\CH_1(\check X)_\Lambda$, as required.
\end{proof}

\smallskip
\begin{cor}\label{cor.isorho}
Under the conditions of Theorem~\ref{thm.res}, the restriction homomorphism
$\rho:\CH_1(X)_\Lambda \to \C(Y)$ is an isomorphism.
\end{cor}

Indeed, $\gamma$ is inverse to $\rho$.

\section{Around the motivic Gysin homomorphism}\label{sec.gy}

The definition of the cycle group $\C(Y)$ of a snc variety $Y$ over the
perfect field $k$  we gave in Section~\ref{sec.cg1} is purely geometric. However we will see
in the next section that under certain assumptions it has a canonical description as a motivic cohomology group. 
In order to construct such an isomorphism we have to use certain explicit descriptions of
the motivic Gysin map for which we did not find a reference in the literature. 
In this section, we can only sketch the arguments, based on \cite{Voe00b}, \cite{Kel13} and \cite{De12}.

\smallskip

In the following, we work in Voevodsky's triangulated category of geometric motives
$DM_{gm}(k):=DM_{gm}(k;\Lambda)$ with coefficients in $\Lambda$.
Recall
that $\Lambda$ is our fixed commutative ring of coefficients.
 We say that
condition $(\dagger)$ is satisfied if one of the following properties holds:

\begin{itemize}
\item[$(\dagger_1)$]  $\Lambda=\Z [1/p]$ or
\item[$(\dagger_2)$] $\Lambda=\Z$ and resolution of singularities holds over $k$ in the sense of \cite[Def.\ 3.4]{FV00}.
\end{itemize}

Under the condition $(\dagger)$, which in this section we will always assume to be satisfied, it is shown in \cite{Voe00b} and \cite{Kel13} that the category
$DM_{gm}(k)$ is a rigid tensor triangulated category. Furthermore,  to any variety
$Z/k$,  one can associate a motive $M_{gm}(Z)\in DM_{gm}(k)$ and a motive with compact
support $M_{gm}^c(Z)\in DM_{gm}(k)$ satisfying certain functorialities. For a smooth equidimensional variety $Z/k$ of
dimension $d$ there is a canonical duality isomorphism
\eq{}{
M_{gm}(Z)^*=M^c_{gm}(Z)(-d)[-2d].
\notag}
For a closed immersion of varieties $Z_1\to Z_2$ there is a canonical exact triangle 
\eq{eq.loccomp}{
M_{gm}^c(Z_1) \to M_{gm}^c(Z_2) \to  M_{gm}^c(Z_2\setminus Z_1) \xrightarrow{\partial} M_{gm}^c(Z_1)[1]. 
}
For a closed immersion of codimension $c$ of smooth equidimensional varieties $Z_1\to Z_2$
there is a dual Gysin exact triangle
\eq{eq.gytr}{
M_{gm}(Z_2\setminus Z_1) \to M_{gm}(Z_2)\xrightarrow{\rm Gy} M_{gm}(Z_1)(c)[2c] \xrightarrow{\partial} M_{gm}(Z_2\setminus Z_1)[1].
}

The following proposition is \cite[Prop. 1.19(iii)]{De12}.

\begin{prop}\label{prop.Gycomp}
Consider a commutative square of varieties
\[
\xymatrix{
Z'_1 \ar[r] \ar[d] &  Z'_2 \ar[d]_{f} \\
Z_1 \ar[r] & Z_2
}
\]
where the horizontal maps are closed immersions of codimension one between smooth varieties. We
assume that in the sense of divisors $f^*(Z_1)=m Z'_1 $. Then the square of Gysin maps
\[
\xymatrix{
  M_{gm}(Z'_2)\ar[r]^-{ m \cdot \rm Gy} \ar[d]_-{f_*} & M_{gm}(Z'_1)(1)[2]\ar[d] \\
 M_{gm}(Z_2)\ar[r]_-{\rm Gy}  & M_{gm}(Z_1)(1)[2]
}
\]
commutes. Here the upper horizontal arrow is $m$ times the Gysin map.
\end{prop}

Recall \cite[p.\ 197]{Voe00b} that Suslin homology $h^S_j(Y)$ of a variety $Y/k$ can be described in terms of motivic homology as
\[
 h^S_j(Y) =\Hom_{DM_{gm}(k)}(\Lambda[j],M_{gm}(Y) ).
\]

\begin{lem}\label{lem.sus0}
There is an
exact sequence
\[
0\to R \to \Zcyc_0(Y)_\Lambda \to h^S_0(Y) \to 0,
\]
where $R$ is the $\Lambda$-module of zero-cycles generated by the divisors of rational
functions $g$, on integral closed curves $C\subset Y$, which satisfy the following
property: Let $\overline C$ be the compactification of $C$ which is normal outside $C$ and let $C^\infty\subset
\overline C$ be the points not lying over $C$. We endow $C^\infty$ with the reduced
subscheme structure. Then we allow those $g$ satisfying $$g\in \ker \big(\sO_{\overline
  C,C^\infty}^\times \to \sO_{C^\infty}^\times\big)\rlap{\text{.}}$$
\end{lem}
 
The lemma is deduced from the definition of Suslin homology, see
\cite[Theorem~5.1]{Sch}.
Note that the rational functions in Lemma~\ref{lem.sus0} are similar to the rational
functions in type 1 data in Section~\ref{sec.cg1}.

\smallskip

For simplicity of notation, we write
\begin{align*}
H^{i,j}(Y) &= \Hom_{DM_{gm}(k)}( M_{gm}(Y),\Lambda(j)[i])\rlap{\text{,}} \\
H_c^{i,j}(Y) &=  \Hom_{DM_{gm}(k)}( M_{gm}^c(Y),\Lambda(j)[i])\rlap{\text{,}}\\
H_{i,j}(Y) &= \Hom_{DM_{gm}(k)}( \Lambda(j)[i], M_{gm}(Y))\rlap{\text{,}}\\
H^c_{i,j}(Y) &= \Hom_{DM_{gm}(k)}( \Lambda(j)[i], M^c_{gm}(Y))\rlap{\text{.}}
\end{align*}
Motivic cohomology has an explicit description in terms of a cdh-cohomology group
\eq{eq.mothcdh}{
H^{i,j}(Y) = H^i(Y_{\rm cdh}, \Lambda(j)), 
}
where $ \Lambda(j)=C_* \Lambda_{tr}(\mathbb G^{\wedge j}_m)[-j]$ is the bounded above complex of
sheaves of $\Lambda$\nobreakdash-modules defined in \cite{SV02}. If $Y/k$ is smooth we can replace the
cdh-topology by the Nisnevich topology in equation~\eqref{eq.mothcdh}.

\smallskip

Let now $Y/k$ be a smooth integral curve with smooth compactification~$\overline Y$.
We write $Y^\infty$ for the reduced closed subscheme $\overline Y\setminus Y$. The Gysin exact triangle~\eqref{eq.gytr}
gives us a morphism $\partial:M_{gm}(Y^\infty)(1)[1] \to M_{gm}(Y)$. Note that there are canonical
isomorphisms
\begin{align}\label{eq.isounit}
\sO(Y^\infty)^\times \otimes_\Z \Lambda  & =
\Hom_{DM_{gm}(k)}(M_{gm}(Y^\infty),\Lambda(1)[1])\\
& =  \Hom_{DM_{gm}(k)}(\Lambda, M_{gm}(Y^\infty)(1)[1]), \nonumber
\end{align}
For the first equality use \cite[Lec.\ 4]{MVW}.
Consider the diagram
\eq{eq.diares}{
\begin{aligned}
\xymatrix{
\sO_{\overline Y,Y^\infty}^\times \ar[r]^-{\alpha} \ar[d]_{\rm div}  & \sO(Y^\infty)^\times \ar[r]^-{\beta} &
\Hom_{DM_{gm}(k)}(\Lambda, M_{gm}(Y^\infty)(1)[1]) \ar[d]^{\partial}\\
\Zcyc_0(Y) \ar[rr] & &   \Hom_{DM_{gm}(k)}(\Lambda,M_{gm}(Y) ) = h^S_0(Y)
}
\end{aligned}
}
where the map $\alpha$ is the restriction map to the product of the residue fields of the
semi-local ring
$\sO_{\overline Y,Y^\infty}$ and $\beta$ is the map induced by~\eqref{eq.isounit}.

\begin{prop}\label{prop.commuGy}
The diagram \eqref{eq.diares} commutes.
\end{prop}

\begin{proof}
Let $V\subset \overline Y$ be an open neighborbood of $Y^\infty$.
We have the duality $h_0^S(Y) = \Hom_{DM_{gm}(k)} (M_{gm}^c(Y),\Lambda(1)[2])$ and the
isomorphism 
\[
\mathcal O_{V}^\times(V)\otimes \Lambda = \Hom_{DM_{gm}} (M_{gm}(V),  \Lambda(1)[1]).
\]
Via these isomorphisms the two morphisms 
\[
\Hom_{DM_{gm}} (M_{gm}(V),  \Lambda(1)[1])  \rightrightarrows  \Hom_{DM_{gm}(k)} (M_{gm}^c(Y),\Lambda(1)[2])
\]
from diagram~\eqref{eq.diares} are by definition both induced by the boundary map of the
homotopy cartesian square
\[
\xymatrix{
M_{gm}(V) \ar[r]\ar[d] & M_{gm}(\overline Y) \ar[d]\\
M_{gm}(V/Y^\infty) \ar[r] & M_{gm}^c(Y)\rlap{\text{.}}
}
\]
Here, for a morphism $Z_1\to Z_2$, the relative motive $M_{gm}(Z_2/Z_1)$ is defined as the mapping cone of
the morphism of complexes of Nisnevich sheaves with transfers $M_{gm}(Z_1)\to M_{gm}(Z_2)$; we view it as an object
of $DM_{gm}(k)$.
The motive with compact support $ M_{gm}^c(Y)$ can be
identified with $M_{gm}(\overline Y/Y^\infty)$, which explains the lower horizontal map in
the square. The square is homotopy cartesian as the cones of the upper and of the lower horizontal
maps are isomorphic to $M_{gm}(\overline Y / V)$.

One way to get the boundary map of the cartesian square is, via the above identification,  as the composition of 
\[
M_{gm}(\overline Y/Y^\infty)  \xrightarrow{} M_{gm}(\overline Y / V)\to M_{gm}(V)[1]\rlap{\text{,}}
\]
which corresponds to the left/lower path in the diagram~\eqref{eq.diares}. Indeed, by the
Gysin isomorphism we get
\[
\Hom_{DM_{gm}} ( M_{gm}(\overline Y / V) , \Lambda(1)[2]) = \bigoplus_{y\in \overline Y\setminus V}
\Lambda .
\]   Another way is
as the composition of
\[
M_{gm}^c(Y) \to M_{gm}(Y^\infty)[1] \to M_{gm}(V)[1]\rlap{\text{,}}
\]
which corresponds to the upper/right path in the diagram~\eqref{eq.diares}. This is clear
in view of the isomorphism~\eqref{eq.isounit}.
\end{proof}

Let $Y$ be equidimensional of dimension $d$ and let $y\in Y^\sm$ be a closed point. The
Gysin morphism and the description of motivic cohomology in terms of Milnor $K$-theory
\cite[Lec.\ 5]{MVW}
induces a morphism
\eq{}{
K^M_{j-d} (k(y)) \xrightarrow{\iota_y} H^{j+d,j}(Y). \notag
}
Here Milnor $K$-theory is taken with $\Lambda$-coefficients. 

\begin{prop}\label{prop.calcmotco}
Let $Y/k$ be a snc variety of dimension $d$ and let $U\subset Y^\sm$ be a dense open
subscheme. 
\begin{itemize}
\item[(i)] For $i-j>d$, the group $H^{i,j}(Y)$ vanishes.
\item[(ii)] For $j\ge d$, the map
\eq{eq.lemcalc}{
\oplus_y \iota_y : \bigoplus_{y\in U_{(0)}} K^M_{j-d} (k(y)) \to  H^{j+d,j}(Y)  
}
is surjective.
\end{itemize}
\end{prop}

\begin{proof}
(i) This is clear from the calculation of $H^{i,j}(Y)$ in terms of cdh-cohomology
as the complex $\Lambda(j)$ lies in $D^{\le j}(Y;\Lambda)$ and the cdh-cohomological
dimension of $Y$ is $d$ \cite[App.]{SV02}.

\smallskip

(ii) We use a double induction on $d$ and on the number $r$ of irreducible components of $Y$.
For $Y=Y_1$ smooth we can use  the coniveau spectral sequence, see \eqref{eq.coniv1}, to get an isomorphism 
\eq{}{
 H^{j+d,j}(Y_1)  \cong A_0(Y_1,j-d)  \notag
}
where the right-hand side is Rost's Chow group with coefficients in Milnor $K$\nobreakdash-theory \cite{R96}.
An elementary moving technique shows that the canonical
map 
\eq{eq.movs}{
\bigoplus_{y\in U_{(0)}} K^M_{j-d}(k(y))\to A_0(Y_1,j-d) 
}
is surjective. In fact, for $\xi \in  K^M_{j-d}(k(x))$ with $x$ closed in $ Y_1\setminus U$,
choose $x'\in U_{(1)}$ such that $x$ lies in the regular locus of $W=\overline{\{x'\}}$.
Choose $ \xi'\in  K^M_{j-d+1}(k(x'))$ such that the residue symbol of $\xi'$ at $x$ is
$\xi$ and such that the residue symbols of $\xi'$ at the other points of $W\cap
(Y_1\setminus U)$ vanish. Now $[\xi]\in A_0(Y_1,j-d)$ is equal to 
$-\sum_{y\in (W\cap  U)_{0} } [\partial_y(\xi')]$, which belongs to the image of \eqref{eq.movs}.

For general $Y$ with irreducible components $Y_1, \ldots, Y_r$, we consider the decomposition $Y=Y_1\cup Y'$ with $Y'=Y_2 \cup \ldots \cup Y_r$. Part of the exact Mayer--Vietoris sequence for
the cdh-covering $Y_1\coprod Y'\to Y$ reads
\eq{}{
H^{j+d-1,j}(Y_1\cap Y') \xrightarrow{\alpha } H^{j+d,j}(Y) \to H^{j+d,j}(Y_1)\oplus H^{j+d,j}(Y') . \notag
}
By the induction assumption, we already know that the image of $\oplus_{y\in U_{(0)}}
\iota_y$ maps surjectively onto the direct sum on the right-hand side. So we only have to show that the
image of $\alpha$ also lies in the image of $\oplus_{y\in U_{(0)}}
\iota_y$.

By the induction assumption, we know that the homomorphism 
\eq{eq.milmot2}{
\iota_y:\bigoplus_{y\in (Y_1\cap Y')^\sm_{(0)}} K^M_{j-d+1}(k(y))\to  H^{j+d-1,j}(Y_1\cap Y')
}
induced by the Gysin map is surjective. We have to give an explicit calculation of its
composition with $\alpha$. 

\begin{claim}
The image of the composition of \eqref{eq.milmot2} and $\alpha$ is contained in the image
of the map \eqref{eq.lemcalc}.
\end{claim}

Let $y$ be a closed point in $(Y_1\cap Y')^{\sm}$.
By the Bertini theorem, we can choose a snc subcurve $C$ on $Y$
containing $y$ such that $C\cap U$ is dense in $C$ (in the case of a finite base field $k$,
use \cite{Po}). Let $C_1=C\cap Y_1$ and $C'=C\cap Y'$. In the triangulated category $DM_{gm}(k)$, we have the commutative diagram of
Gysin maps
\begin{align}
\begin{aligned}
\small
\label{eq.comcur}
\xymatrix@C=1.3em{
M_{gm}(Y_1\cap Y')
\ar[r]\ar[d]_{\rm Gy}    & M_{gm}(Y_1) \oplus M_{gm}(Y') \ar[r]\ar[d]_{\rm Gy} &
M_{gm}(Y) \ar[r]^-{\partial} \ar[d]_{\rm Gy} & \\ 
M_{gm}(C_1\cap C')(c)[2c]\ar[r] & \left(M_{gm}(C_1) \oplus M_{gm}(C')\right)(c)[2c] \ar[r] &    M_{gm}(C)(c)[2c] \ar[r]  &
}
\end{aligned}
\end{align}
where $c=d-1$. These Gysin maps are constructed via the Gysin maps for smooth schemes by
using the Cech simplicial scheme associated to the covering by irreducible components. This shows that we can assume $d=1$ in the proof of the claim. 
For $d=1$, one gets a commutative
diagram 
\[
\xymatrix{
K^M_j(\sO_{Y_1,Y_1\setminus (U\cap Y_1)} ) \ar[d] \ar@{->>}[r]  & K^M_j(Y_1\cap Y') \ar[d]^\wr\\
  H^{j,j}(\sO_{Y_1,Y_1\setminus (U\cap Y_1)} ) \ar[r] \ar[d]_{\partial}  & H^{j,j}(Y_1\cap
Y') \ar[d]_{\alpha} \\
\bigoplus_{y\in (U\cap Y_1)_{(0)}}  H^{j-1,j-1}(y) \ar[r]_-{\oplus {\rm Gy}_y}  &  H^{j+1,j}(Y)\rlap{\text{.}}
}
\]
Here, the Milnor $K$-group of a ring is simply the quotient of the tensor algebra over the
units by the Steinberg relations and the maps from Milnor $K$-theory to motivic cohomology
are as defined in \cite[Lec.~5]{MVW}.
The lower square is commutative because for an open neighborhood $V$ of $Y_1\cap Y'$ in $Y_1$,
 the composition of $\partial$ from \eqref{eq.comcur} with $M_{gm}(Y_1\cap Y')\to M_{gm}(V)$
is equal to the composition of the maps
\[
M_{gm}(Y)   \to M_{gm}(Y/(V\cup Y')) \xleftarrow{\sim} M_{gm}(Y_1/V) \xrightarrow{\partial_{(Y_1,V)} } M_{gm}(V)[1]\rlap{\text{.}}
\]
We now explain the isomorphism in the middle: 
by Mayer--Vietoris applied  
 to the 
covers $Y=Y_1\cup Y'$ and $V\cup Y'$  (\cite[Prop.~4.1.3]{Voe00b}, \cite[Prop.~5.5.4]{Kel13}), one has a commutative diagram in which the rows are exact triangles  
\[\xymatrix{ \ar[d] M_{gm}(V) \ar[r] & \ar[d]  M_{gm}(V\cup Y') \ar[r] &  \ar@{=}[d] M_{gm}(Y'/Y'\cap Y_1) \\
\ar[d] M_{gm}(Y_1) \ar[r] & \ar[d]  M_{gm}(Y_1\cup Y') \ar[r] & M_{gm}(Y'/Y'\cap Y_1)\\
M_{gm}(Y_1/V) \ar[r]^-{ \sim} & M_{gm}(Y/V\cup Y')\rlap{\text{.}}
}
\]
 In addition, by definition of  $M_{gm}(Y_1/V)$ and $M_{gm}(Y/V\cup Y')$, the columns are exact. This implies the isomorphism. 
This finishes the proof of
the claim and also of Proposition~\ref{prop.calcmotco}.
\end{proof}

Let now $X/k$ be a connected smooth variety and $Y\subset X$ a snc
divisor. 
For a closed point $y\in Y^\sm$, let
\eq{eq.lamd}{
\lambda_y:k(y)^\times \to h^S_0(X\setminus Y)
}
be the composite homomorphism
\[
k(y)^\times\xrightarrow{\sim} H^{1,1}(y) =H_{1,1}(y) \to H_{1,1}(Y^\sm) \xrightarrow{\partial} h^S_0(X\setminus Y). 
\]

\begin{prop}\label{prop.lamde}
The maps $\lambda_y$ ($y\in Y^\sm$) are uniquely characterized by the following property.
Let $C\subset X$ be an integral closed curve disjoint
from $Y^\sing$
and not contained in~$Y$. Let $C^\infty=(C\cap Y)_{\rm red}$ and let $\eta$ be the generic
point of $C$. For $y\in C^\infty$,
let $m_y$ be the intersection multiplicity of $C$ and $Y$ at $y$. Consider an element
$g\in \sO_{C,\{\eta \}\cup C^\infty}^\times $. Then we have
\eq{eq.propcu}{
{\rm div}(g) + \sum_{y\in C^\infty} m_y \lambda_y(g|_{y}) =0 \quad \text{ in } \quad
h^S_0(X\setminus Y).
}
\end{prop}

\begin{proof}
Uniqueness is clear:
take a curve $C$ intersecting $Y$ transversally and a rational function $g$ equal to~$1$
at all but one point $y$ of~$C^\infty$, this determines the map~$\lambda_y$.
To prove~\eqref{eq.propcu},
we first reduce to the case $\dim(X)=1$ by applying Proposition~\ref{prop.Gycomp} with $Z_1=Y^\sm$,
$Z_2=X\setminus Y^\sing$ and $Z'_2$  the normalization of~$C$.
The case $\dim(X)=1$ is a consequence of Proposition~\ref{prop.commuGy}.
\end{proof}

\section{Motivic interpretation of $\C(Y)$} \label{sec.motive}

For the category $DM_{gm}(k)$ we use the same notation as in Section~\ref{sec.gy}.
In this section $k$ is a perfect field,
$Y/k$ is a projective  snc variety of dimension $d$ and $\tau:Y\to \mathbb N$ is
the function locally counting  irreducible components. Let $p$ be the exponential
characteristic of $k$.

\smallskip 

The localization exact triangle 
\ga{}{
M_{gm}(Y^\sing) \to M_{gm}(Y) \to M_{gm}^c(Y^\sm) \to   M_{gm}(Y^\sing)[1] \notag
}
from~\cite[Prop.~4.1.5]{Voe00b}, \cite[Prop.~5.5.5]{Kel13} induces an exact sequence 
\ga{eqn:pres}{
 H^{2d-1,d}(Y^\sing)    \to h^S_0(Y^\sm) \to H^{2d,d}(Y) \to  0.
}
In fact $h^S_0(Y^\sm) =H^{2d,d}_c(Y^\sm)$ by duality.
 We remark that  Proposition~\ref{prop.calcmotco} implies that $H^{2d,d}(Y^\sing)=0$ and that $
H^{2d-1,d}(Y^\sing)$ is spanned by the images of the groups $k(y)^\times\otimes_\Z \Lambda$ with $\tau(y)=2$ via the
Gysin maps. Let $Y_1,\ldots , Y_r$ be the irreducible components of $Y$ and assume that
$y$ is contained in the two components $Y_1$ and $Y_2$. The composition $\lambda_y$ of the maps
\[
k(y)^\times \to H^{2d-1,d}(Y^\sing) \to h^S_0(Y^\sm) =\oplus_{i=1}^r h^S_0(Y_i\setminus
(Y_i\cap Y^\sing)) 
\]
is equal to the sum of the two maps defined in \eqref{eq.lamd} for the two components $Y_1$
and $Y_2$. This is clear by using functoriality along the map $Y_1\coprod Y_2 \to Y$. 
Thus, the exact sequence \eqref{eqn:pres} induces an exact sequence
\eq{}{
\bigoplus_{\tau(y)=2} k(y)^\times \otimes_\Z \Lambda \xrightarrow{\oplus_y \lambda_y}  h^S_0(Y^\sm) \to H^{2d,d}(Y) \to  0.
\notag }

By Lemma~\ref{lem.sus0}, we can identify the quotient of $\Zcyc_0(Y^\sm)_\Lambda$ by  the
sub-$\Lambda$-module generated by type 1 relations with
$h^S_0(Y^\sm)$ (see Section~\ref{sec.cg1}). Therefore, the cycle group
$\C(Y)$ has the analogous presentation
\eq{}{
\bigoplus_{(C_2,g_2)} \Lambda \xrightarrow{\rm div} h^S_0(Y^\sm) \to \C(Y) \to  0. \notag
}
The direct sum on the left runs over all type 2 data.

\smallskip

The canonical surjective map 
\[
\theta: \bigoplus_{(C_2,g_2)} \Lambda \to \bigoplus_{\tau(y)=2} k(y)^\times \otimes_\Z \Lambda
\]
which sends $1\in \Lambda_{(C_2,g_2)}$ to $g_2|_{C_2\cap Y^\sing}$ gives rise to the left
square in the
diagram
\eq{}{
\xymatrix{
\bigoplus_{(C_2,g_2)} \Lambda \ar[r]^{\rm div} \ar[d]_{\theta} & h^S_0(Y^\sm) \ar[r] \ar@{=}[d] &
\C(Y) \ar[r]  \ar@{-->}[d]^{\wr}_{\Gamma} & 0\\
\bigoplus_{\tau(y)=2} k(y)^\times \otimes_\Z \Lambda \ar[r]^-{\oplus_y \lambda_y}&
h^S_0(Y^\sm) \ar[r] & H^{2d,d}(Y) \ar[r] &  0\rlap{\text{.}}
} \notag
}
In fact, the left square commutes by Proposition~\ref{prop.lamde}, so there is a unique
isomorphism $\Gamma$ making the diagram commutative. In summary, we have shown

\begin{thm}\label{thm.motcomp}
Assume that the ring of coefficients~$\Lambda$ satisfies condition $(\dagger)$ from Section~\ref{sec.gy}.
For $Y/k$ a projective snc variety, there is a unique isomorphism $\Gamma$ making the
diagram
\[
\xymatrix{
\Zcyc_0(Y^\sm)\ar[d] \ar[rd] & \\
\C(Y) \ar[r]^-{\Gamma}  & H^{2d,d}(Y)
}
\]
commutative. The diagonal map is induced by the Gysin maps of closed points $y\in Y^\sm$ as in Proposition~\ref{prop.calcmotco}(ii)  for $j=d$.
\end{thm}

\begin{rmk}\label{rmk.fico}
Theorem~\ref{thm.motcomp} remains true for $\Lambda=\Z/m\Z$ with $m$ prime to~$p$. To see
this, simply tensor the diagram taken with $\Lambda=\Z[1/p]$ coefficients  with $\Z/m\Z$
and use cohomological dimension to show that
\[H^{2d,d}(Y,\Z[1/p]) \otimes_{\Z[1/p]} \Z/m\Z \xrightarrow{\sim}  H^{2d,d}(Y,\Z/m\Z)
\]
is an isomorphism.
\end{rmk}

\begin{rmk} \label{rmk:GammaforU}
The natural homomorphism $\Zcyc_0(U)\to \sC(Y)$ is surjective
for any dense Zariski open subset $U\subset Y^{\rm sm}$,
as can be seen by moving zero-cycles on $Y^{\rm sm}$ by type 1 relations.  
\end{rmk}

We expect that the condition of projectivity in Theorem~\ref{thm.motcomp} can be weakened
to quasi-projectivity. 

\section{Proof of Theorem~\ref{thm.res}}  \label{sec:proof}

This section is devoted to the proof of Theorem~\ref{thm.res}, which we subdivide
according to the extra assumptions. All constructions are based on the isomorphism $\C(Y)
\cong H^{2d,d}(Y)$ of Theorem~\ref{thm.motcomp}. 

\subsection{Proof of Theorem~\ref{thm.res} (i) and (ii)}
The assumption is that the  (perfect) residue field $k$ is separably  closed  (thus algebraically closed), in (i), or  finite, in (ii).

By
Corollary~\ref{cor:isoclosed}, the \'etale
cycle map $ \C (Y)\xrightarrow{\sim} H^{2d}(Y_\et , \Lambda (d))$ is an isomorphism.
Define $\rho$ such that the diagram
\[
\xymatrix{
\CH_1(X)_\Lambda \ar[r] \ar@{-->}[d]_{\rho}  & H^{2d}(X_\et , \Lambda (d)) \ar[d]^{\wr}\\
\C(Y)  \ar[r]^-{\sim} & H^{2d}(Y_\et , \Lambda (d))
}
\]
is commutative. Here the upper horizontal map is the \'etale cycle map on $X$.

\subsection{Proof of Theorem~\ref{thm.res} (iv)}

Let $\Lambda=\Z/m\Z$ with $m$ prime to $p$.
Recall that the Milnor $K$-sheaf with $\Lambda$\nobreakdash-coefficients is defined as the
Nisnevich sheafification of the presheaf on affine schemes $\Spec (R)$ given as follows: take the quotient of
the tensor algebra \[ R\mapsto \Lambda \otimes_\Z \bigoplus_{i\in \mathbb N}
\underbrace{R^\times \otimes_\Z \cdots \otimes_\Z R^\times 
}_{i \text{ times}} \]
by the two-sided ideal generated by the elements $a\otimes (1-a)$ with $a,1-a\in R^\times$.

 Rigidity
for Milnor $K$-theory is the statement that the stalk of the  Milnor $K$-sheaf in the Nisnevich
topology at a point
$y\in Y$ satisfies \[(\sK^M_{Y,j})_y\xrightarrow{\sim} K^M_j(k(y)).\]
The same rigidity property holds for motivic cohomology. More precisely, consider the
morphism of sites
$\epsilon:Y_{\rm cdh} \to Y_{\rm Nis}$
and the motivic complex in the Nisnevich topology defined as $\Lambda(j)_{Y,\rm Nis}=R
\epsilon_* \Lambda(j)_Y$.
Then, for $y\in Y$, the map
$\sH^i(\Lambda(j)_{Y,\rm Nis})_y \to   H^{i,j}(y)$
is an isomorphism. 
For smooth $Y$, this follows from \cite[Cor.\ 0.4]{HY}.
For a general snc
variety $Y$ with irreducible components $Y_1,\ldots, Y_r$,
one reduces the statement to the smooth case by means of the descent spectral sequence of the cdh-covering \[\coprod_{i=1}^r Y_i \to
Y. \]
Thus, using the isomorphism between Milnor $K$-theory and motivic cohomology for
fields~\cite[Lec.\ 5]{MVW}, we obtain
\begin{prop}
For $\Lambda=\Z/m\Z$ with $m$ prime to $p$, there is a canonical isomorphism of Nisnevich
sheaves
\[
S:\sK^M_{Y,j} \xrightarrow{\sim } \sH^j(\Lambda(j)_{Y,\rm Nis}).
\]
For $i>j$, the sheaf $\sH^i(\Lambda(j)_{Y,\rm Nis})$ vanishes.
\end{prop}

Combining this proposition with Theorem~\ref{thm.motcomp} and cohomological dimension, we see that there is an isomorphism
\eq{eq.isomil}{
\C(Y) \xrightarrow{\sim} H^d(Y_{\Nis} ,\sK^M_{Y,d}).
}

As $A$ is equicharacteristic, the Gersten conjecture for Milnor $K$-theory \cite{Ker09}
implies that there is a canonical isomorphism  \[\CH_1(X)_\Lambda \xrightarrow{\sim} H^{d}(X_\Nis,\sK^M_{X,d}  ). \]
Define $\rho$ such that the diagram
\[
\xymatrix{
\CH_1(X)_\Lambda \ar[r]^-{\sim} \ar@{-->}[d]_{\rho}  &  H^{d}(X_\Nis,\sK^M_{X,d}  ) \ar[d]\\
\C(Y)  \ar[r]^-{\sim} &  H^{d}(Y_\Nis,\sK^M_{Y,d}  )
}
\]
is commutative. The right vertical arrow is the restriction homomorphism.

 \subsection{Proof of Theorem~\ref{thm.res} (iii)}
In view of Remark~\ref{rmk.reiii}, we can assume that $\Lambda=\Z[1/(p (d-1)!)]$.
As $X$ is regular, its $K$-theory of vector bundles $K_0(X)$ is the same as its
$K$-theory of coherent sheaves
$K'_0(X)$ \cite[Sec.~2.3.2]{Gil05}.  Let $F_0\subset K'_0(X)$ be the subgroup spanned by
classes of coherent sheaves supported in dimension $0$.

\begin{claim}
We have $F_0=0\subset K'_0(X)$.
\end{claim}

\begin{proof}
By linearity, it is enough to see that the class of the skyscraper sheaf $i_{x,*} k(x)$ in
$K'_0(X)$ is zero, where $x$ is a  closed point of $Y$ and $i_x:x\to X$ is the
immersion.  For simplicity of notation, we write $k(x)$ for this  coherent sheaf of $\sO_X$-modules.
 Choose a regular connected  closed subscheme $Z\subset X$ of dimension~$1$ with
$x\in Z$ and $Z/A$ flat, $Z=\Spec R$. Let $\pi_Z\in R$ be a prime element. The short
exact sequence of coherent sheaves of $\sO_X$-modules
\[
0\to \sO_Z  \xrightarrow{ \pi_Z} \sO_Z \to k(x) \to 0
\]  
implies $[k(x)]=0\in K'_0(X)$.
\end{proof}

Let $\alpha$ be the composite homomorphism \[\CH_1(X)\to
K'_0(X)/F_0=K_0'(X)\xleftarrow{\sim} K_0(X) \]
 where the first map assigns to the cycle class of 
an integral curve the class of its structure sheaf \cite[Thm.~34]{Gil05}. 
  As $K$-theory is contravariant, one has a restriction homomorphism $K_0(X)\to K_0(Y)$.  
  
 \smallskip
 
From the Grothendieck construction based on the projective bundle formula, one obtains a
Chern class map $c_d: K_0(Y)\to  H^{2d,d}(Y)$, see \cite[Sec.\ V.11]{W13}. 
Let $\rho'$ be the  map
\eq{eq.defrhok}{
 \CH_1(X) \xrightarrow{\alpha} K_0(X)\to K_0(Y)\xrightarrow{c_d}  H^{2d,d}(Y)\xrightarrow{\sim}  \C(Y).
}

The following claim finishes the proof of Theorem~\ref{thm.res}(iii).
  
\begin{claim} \label{claim:cd}
The map \eq{eq.clreK}{ \Zcyc^g_1(X)\to \CH_1(X) \xrightarrow{\rho'} \C(Y)}  is additive and is  $(-1)^{d-1}(d-1)!$ times the map induced by the restriction homomorphism $\tilde \rho: \Zcyc_1^g(X) \to \Zcyc_0(Y^\sm)$.
\end{claim}

Claim~\ref{claim:cd}  implies
that the map $\rho$ defined as $\rho'$ tensored by $\Lambda$ and divided by
$(-1)^{d-1}(d-1)!$ satisfies the properties required by Theorem~\ref{thm.res}.

\begin{proof}[Proof of Claim~\ref{claim:cd}]
We extend the argument of \cite[proof of Prop.~2]{BS98}.
First, we show additivity of the composite map \eqref{eq.clreK}. Then we prove the claim
for an integral cycle in $\Zcyc_1^g(X)$.

\smallskip

Let $F_0^\sm \subset K_0(Y)$ be the subgroup generated by the classes $[k(x)]$ with $x\in
Y^\sm$ of dimension $0$. Note that the coherent sheaf $k(x)$ of $\sO_Y$-modules has finite
projective dimension, so the element $[k(x)]$ is well-defined in $K_0(X)$. We prove that $c_i(w)$ vanishes in $H^{2i,i}(Y)$ for $w\in
F_0^\sm$ and $i<d$. Indeed,
there is an open subscheme $U\subset Y$ such that $Y\setminus U$ consists of finitely many
closed points in $Y^\sm$ and such that $w|_U=0$. Now the Gysin exact sequence 
\[
0= H^{2i-2d,i-d}(Y\setminus U)  \to H^{2i,i}(Y)\to H^{2i,i}(U)
\]
and the vanishing of $c_i(w)|_{U}$ imply the vanishing of $c_i(w)$. 
It follows, by the Whitney sum formula, that $c_d|_{F_0^\sm}$ is additive.
As
the image of $\Zcyc_1^g(X)$ in $K_0(Y)$ is contained in $F_0^\sm$, the composite map
\eqref{eq.clreK} is additive as well.

\smallskip

Now we prove the claim for
$Z$ an integral cycle in $\Zcyc^g_1(X)$.
Let $y$ be the point  $(Z\cap Y)_{\rm red} $ and let $n$ be the intersection
multiplicity of $Z$ and $Y$. Let~$Y_1$ be the irreducible component of $Y$ containing $y$.

The commutative diagram with exact columns (Gysin sequence; note that
$H_y^{2d,d}(Y)=H_y^{2d,d}(Y_1)=\Lambda$ as $y$ is a non-singular point)
\ga{}{\xymatrix{ 
  &  & \Lambda \ar[d]_{{\rm Gy}_y}    \ar@{=}[r]   & \Lambda \ar[d]_{{\rm Gy}_y}\\
K_0(X) \ar[d] \ar[r] &
   K_0(Y) \ar[d] \ar[r]^-{c_d} &   H^{2d,d}(Y) \ar[d] \ar[r] & H^{2d,d}(Y_1) \ar[d] \\
 K_0(X\setminus Z)  \ar[r] & K_0(Y\setminus \{y\}) \ar[r]^-{c_d} &  H^{2d,d}(Y\setminus
\{y\}) \ar[r] &  H^{2d,d}(Y_1\setminus \{y\}) 
} \notag}
shows that the image $\gamma$ of $Z\in \Zcyc^g_1(X)$ in $H^{2d,d}(Y)$ satisfies $\gamma|_{Y\setminus \{y\}}=0$. 
Thus $\gamma={\rm Gy}_y(n')$ for a unique $n'\in \Lambda$. 
Note that the Gysin maps in the above diagram are injective. In fact, the composition of the
Gysin map on the right with the  degree map
$H^{2d,d}(Y_1)\cong \CH_0(Y_1) \xrightarrow{\rm deg} \Lambda$ is multiplication by
$[k(y):k]$, which is injective (recall that $\Lambda=\Z[1/(p (d-1)!)]$).

In the final step, we give a different description of $\gamma|_{Y_1}$.
The image   of $Z$ in $K_0(Y_1)$ is $n [k(y)]$ by \cite[p.~255]{Gil05}.
We have a commutative diagram
\[
\xymatrix{ \ar[d] K_0(Y) \ar[r]^-{c_d} &\ar[d] H^{2d,d}(Y)\\
 K_0(Y_1) \ar[r]^-{c_d} & H^{2d,d}(Y_1)\cong \CH_0(Y_1)\rlap{\text{.}}
 }
\]
 By the non-singular
Grothendieck--Riemann--Roch theorem \cite[Sec.\ 15.3]{Ful84} the bottom Chern class satisfies  \[c_d([k(y)]) = {\rm
  Gy}_y((-1)^{d-1} (d-1)!) \] in $H^{2d,d}(Y_1)$. 
As $n c_d([k(y)])=\gamma|_{Y_1}$ we get $n'=n (-1)^{d-1} (d-1)!$.
\end{proof}

\section{Special fields}

Let $Y/k$ be a proper snc variety of dimension $d$ and let $\Lambda =\Z/m\Z $ with $m$
prime to $p$.
In this section, we calculate the group $\C(Y)$ in terms of \'etale cohomology when $k$ is finite or algebraically
closed. This calculation shows that for these special fields, our main result, Theorem~\ref{thm.main}, is equivalent to
results of Saito--Sato and Bloch \cite{SS10}, \cite[App.]{EW13}.

\smallskip

We write $H_\et^{i,j}(Y)$ for the \'etale cohomology group $H^i(Y_\et , \Lambda (j))$.
Recall that there is a functorial \'etale realization homomorphism \[R_\et^{i,j}: H^{i,j }(Y)
\to H_\et^{i,j}(Y). \]
For smooth varieties, it is discussed in \cite[Thm.\
10.3]{MVW}.  For snc varieties, one can define it via the smooth Cech simplicial scheme
corresponding to the covering of $Y$ by its irreducible components.  

\begin{prop}\label{prop.compet}  Let $k$ be a finite field or an algebraically closed field
  and let  $\Lambda =\Z/m\Z $ with $m$
prime to $p$.
For a snc variety $Y/k$ and
 $s\ge d$, $i\in \Z$,  the realization map $R^{i,s}_\et$ is an isomorphism, except possibly
when $k$
is finite, $i=2d+1$ and  $s=d$.
\end{prop}

\begin{proof} We give the proof for $Y/k$ smooth first and then deduce the general case by
 a Mayer--Vietoris sequence.

\smallskip

{\it Case $Y/k$ smooth:} 
Consider the coniveau spectral sequences, respectively in the Nisnevich and in the \'etale topology:
\begin{align}\label{eq.coniv1}
 E_{1}^{i,j}(s) =  \bigoplus_{y\in Y^{(i)}} H^{j-i,s-i}(y) &\Rightarrow
 H^{i+j,s}(Y)\rlap{\text{,}}\\
 \label{eq.coniv2} E_{1,\et}^{i,j}(s) =  \bigoplus_{y\in Y^{(i)}} H^{j-i,s-i}_\et(y) &\Rightarrow
 H^{i+j,s}_\et(Y)\rlap{\text{.}}
\end{align}
Recall the following facts.
\begin{itemize}
\item[(i)](Beilinson--Lichtenbaum conjecture) For any field $F$ of characteristic prime to $m$ and any $i\le j$, the realization map \[R_\et^{i,j}: H^{i,j}(F)
  \to H^{i,j}_\et(F) \]
is an isomorphism \cite{Voe11}.
\item[(ii)] $E_1^{i,j}(s)=0$ for $j>s$, by the definition of the motivic complex \cite[Lec.~3]{MVW}.
\item[(iii)](Cohomological dimension) 
  $E_{1,\et}^{i,j}(s)=0  $ for $j>d+1$ if $k$ is finite, for $j>d$  if $k$ is algebraically
  closed. 
\item[(iv)](Kato conjecture)  For $k$ finite and $i\ne d$, we have $E_{2,\et}^{i,d+1}(d)=0$
  \cite{KS12}.
\end{itemize}
Comparing the spectral sequences and using (i)-(iv), we get that
the \'etale realization map
induces an isomorphism
\[E_{\infty}^{i,j}(s)\xrightarrow{\sim} E_{\infty,\et}^{i,j}(s)  \]
if $s\ge d$ except possibly if $s=i=d$ and $j=d+1$. 
This proves
Proposition~\ref{prop.compet} for $Y$ smooth.

\smallskip

{\it Case $Y/k$ snc variety:} 
Let $Y=Y_1\cup \dots \cup Y_r$ be the decomposition into irreducible components and set
$Y'= Y_2\cup \dots \cup Y_r$. We proceed by induction on $r$. The case $r=1$ is shown above,
so assume $r>1$. Consider the \'etale realization morphism of exact Mayer--Vietoris sequences
\begin{align*}
\xymatrix{
H^{i-1,s}(Y_1)\oplus H^{i-1,s}(Y') \ar[d] \ar[r] &  H^{i-1,s}(Y_1\cap Y') \ar[d]  \ar[r] & H^{i,s}(Y)
\ar@{-}[r] \ar[d]  &  \\
H^{i-1,s}_\et(Y_1)\oplus H_\et^{i-1,s}(Y')  \ar[r]  &  H_\et^{i-1,s}(Y_1\cap Y')   \ar[r] & H_\et^{i,s}(Y)
\ar@{-}[r] &
}
 \\
\xymatrix{
\ar[r] & H^{i,s}(Y_1)\oplus H^{i,s}(Y')\ar[r]\ar[d]  & H^{i,s}(Y_1\cap Y')\ar[d]  \\
\ar[r] & H^{i,s}_\et(Y_1)\oplus H^{i,s}_\et(Y')\ar[r] & H^{i,s}_\et(Y_1\cap Y')\rlap{\text{.}}
}
\end{align*}
By the five lemma, the middle vertical map is an isomorphism, since, by the induction assumption, the other vertical
maps are isomorphisms.
\end{proof}

For $k$ algebraically closed, consider the composite map
\eq{eq.degmap}{
\deg: \C(Y) \to \bigoplus_{1\le i\le r} \CH_0(Y_i)_\Lambda \to \Lambda^{r}
}
in which the first map
is induced by the canonical maps $\Zcyc_0(Y^\sm)\to \Zcyc_0( Y_i)$ and the second
map is the direct sum of the usual degree maps of zero-cycles of the irreducible
components of $Y$.

Combining Proposition~\ref{prop.compet}  with Theorem~\ref{thm.motcomp}, we get

\begin{cor}\label{cor:isoclosed}
For $k$ finite or algebraically closed
and $Y/k$ a projective snc variety,
the \'etale cycle map $\C(Y)\to H^{2d,d}_\et(Y)$ is
an isomorphism.

For $k$ algebraically closed and $Y/k$ a projective snc variety, the degree map \eqref{eq.degmap} is
an isomorphism.
\end{cor}

As a further application, we discuss zero-cycles over certain $1$-dimensional and
$2$-dimensional local fields. Over $1$-dimensional local fields, this reproves a result of
\cite{SS10}.

\begin{prop}[Saito--Sato] \label{prop.filo}
Let $K$ be a local field with finite residue field~$k$ of characteristic $p$.
Let $\Lambda=\Z/m\Z$ with $m$
prime to $p$. For any variety $X_K/K$, the group $\CH_0(X_K)_\Lambda$ is finite.
\end{prop}

\begin{proof}
By a simple resolution technique and the Gabber--de Jong alteration theorem \cite[Thm.\ 0.3]{ILO}, we can assume, without loss of generality,
that there is a regular scheme $X$ which is projective and flat over $A$ such that the
reduced special fiber $Y$ of $X$ is a simple normal crossings divisor and such that the
generic fiber of $X$ is $X_K$.
Here $A$ is the ring of integers in $K$.
The left-hand side map in
\[
\CH_0(X_K)_\Lambda \twoheadleftarrow \CH_1(X)_\Lambda \xrightarrow{\sim} \C(Y)
\]
is surjective and the right-hand side map is an isomorphism by Theorem~\ref{thm.main}.
We conclude by the isomorphism $\C(Y) \xrightarrow{\sim} H^{2d}(Y_\et,
\Lambda(d))$ from Corollary~\ref{cor:isoclosed} and the finiteness of \'etale cohomology.
\end{proof}

\begin{prop}\label{prop.laur}
Let $K$ be the field of formal Laurent power series  $  k((\pi))$ with $[k:\Q_p]<\infty$, $A=k[[\pi]]$. 
For any variety $X_K/K$ and for $\Lambda=\Z/m\Z$ with $m$ prime to $p$, the
group $\CH_0(X_K)_\Lambda$ is finite.
\end{prop}

\begin{proof}
We can assume, without loss of generality, that $\Lambda=\Z/\ell\Z$ for a prime~$\ell$.
As the map $\CH_0(X_K)_\Lambda\to \CH_0(X_{X_{K(\boldsymbol{\mu}_\ell)}})_\Lambda $ is injective, we can
also assume that $A$ contains a
primitive $\ell$-th root of unity, which we fix from here on.
By a simple resolution technique and the Gabber--de Jong alteration theorem \cite[Thm.\ 0.3]{ILO}, we can assume
that there exists a regular scheme $X$, projective and flat over $A$,
whose
reduced special fiber $Y/k$ is a simple normal crossings divisor and
whose
generic fiber is $X_K$.

The left-hand side map in
\[
\CH_0(X_K)_\Lambda \twoheadleftarrow \CH_1(X)_\Lambda \xrightarrow{\sim} \C(Y)
\]
is surjective and the right-hand side map is an isomorphism by Theorem~\ref{thm.main}.

Let $d=\dim(Y)$.
Note that $\C(Y)\cong H^{2d,d}(Y)$ by Theorem~\ref{thm.motcomp}. By a
Mayer--Vietoris sequence argument for a closed covering similar to the proof of Proposition~\ref{prop.compet}, we reduce
to showing that $H^{d_W+d,d}(W)$ is finite for any smooth projective variety $W/k$ of dimension
$d_W\le d$.

The coniveau spectral sequence for $W$, see \eqref{eq.coniv1}, degenerates to an
isomorphism 
\[
H^{d_W+d,d}(W) \cong {\rm coker}\mkern3mu [ \oplus_{w\in W_{(1)}} K^M_{d-d_W+1}(k(w)) \to  \oplus_{w\in W_{(0)}} K^M_{d-d_W}(k(w)) ]\rlap{\text{,}}
\]
from which we deduce:
\begin{itemize}
\item[(i)] For $d_W<d-2$, we have $ H^{d_W+d,d}(W)=0$.
\item[(ii)] For $d_W=d-2$, there is an isomorphism  $$H^{d_W+d,d}(W)\cong H_0(KC^{(0)}(W))$$ onto
  Kato homology with $\Lambda$-coefficients, see e.g.\ \cite{KS12}.
\item[(iii)] For $d_W=d-1$, there is an isomorphism $H^{d_W+d,d}(W)\cong SK_1(W)_\Lambda$, see
  \cite{Fo15}.
\item[(iv)] For $d_W=d$, there is an isomorphism $H^{d_W+d,d}(W)\cong \CH_0(W)_\Lambda$.
\end{itemize}
Indeed, for a finite extension $F/\Q_p$ the Milnor $K$-group with finite coefficients $K^M_i(F)$ vanishes for $i>2$ \cite[Prop.\
VI.7.1]{W13}, hence~(i).
For~(ii), we note that by the Bloch--Kato conjecture
\cite{Voe11}
and the fixed $\ell$-th root of unity,
there is an isomorphism $K^M_i(F)\cong
H^{i}(F_\et,\Lambda(i-1))$ for any field $F/k$.

It is known, see the introduction to \cite{KS12} and \cite[Rmk.~5.2]{Fo15}, that the groups
$SK_1(W)_\Lambda$ and $H_0(KC^{(0)}(W))$ are finite. Indeed, for the latter, we can assume,
without loss of generality (use~\cite{ILO}, Theorem~X.2.4), that there is a regular model
$\mathcal W$
of $W$ over $\mathcal O_k$ whose reduced closed fibre is a snc divisor. Then, by
\cite[Thm.~0.4]{KS12}, the group $H_0(KC^{(0)}(W))$ is isomorphic to a finitely generated
combinatorial homology group which depends on the closed fibre of~$\mathcal W$.
The finiteness of $\CH_0(W)_\Lambda$ is shown in Proposition~\ref{prop.filo}.
\end{proof}

\section{Open problems}\label{sec.prob}

Let us keep the notation of Section~\ref{sec:rest}
and consider the ring of coefficients $\Lambda=\Z/m\Z$ with $m$ arbitrary, i.e.\ we
allow $p|m$ in positive characteristic. 
 Set $Y_n=X\otimes_A A/(\pi^n)$, where $\pi$ is a
prime element of the discrete valuation ring $A$. Recall that $K$ denotes the quotient
field of $A$.   
The isomorphism~\eqref{eq.isomil} motivates us to define an ad hoc version of `motivic cohomology' of $Y_n$ in bidegree
$(2d,d)$ as 
\[
\C (Y_n)= H^d(Y_\Nis,\mathcal K^M_{Y_n,d} ),
\]
where $K^M_{Y_n,d}$ is the Milnor $K$-sheaf with $\Lambda$-coefficients as defined in Section~\ref{sec:proof}.

Assuming the Gersten conjecture for Milnor $K$-theory of $X$, we get a compatible system of homomorphism
\eq{eq.modp}{
\CH_1(X)_\Lambda \cong H^d(X_\Nis, \mathcal K^M_{X,d})\to  \C(Y_n) \quad \quad (n\ge 1).
}
It is not difficult to show that the pro system $(\C(Y_n))_n$ is pro constant if $\ch(K)=0$.
As a potential generalization of our main result, Theorem~\ref{thm.main}, we conjecture:

\begin{conj}\label{conj.modp}
For $\ch(K)=0$,
the map \eqref{eq.modp} is an isomorphism of pro abelian groups. Here we think of the
group $\CH_1(X)$ as a constant pro group.
\end{conj}

Conjecture~\ref{conj.modp} is related to a conjecture of
Colliot-Th\'el\`ene \cite{CT95}. In fact, we expect that for $\Lambda=\Z$ and
$[K:\Q_p]<\infty$, there should be an  isomorphism 
\[\CH_0(X_K)\cong \Z\oplus \Z_p^w \oplus (\text{finite group}) \oplus (\text{divisible group}).\] 
Here, the integer $w$ should be the $\Z_p$-dimension of the Lie group of $K$\nobreakdash-points of the Albanese variety of $X_K$. The
decomposition is not canonical, but the individual summands are unique up to isomorphism.

\smallskip

In another direction, we expect that an analogue of Theorem~\ref{thm.main} for ``higher zero-cycles''
holds. For a noetherian scheme $S$ of finite Krull dimension, let
$H^{i,j}(S) $ denote motivic cohomology with $\Lambda$-coefficients,  defined for example using
the Eilenberg--MacLane spectrum constructed in \cite{Sp12}. 
We consider the restriction homomorphism of motivic cohomology groups
\eq{eq.reij}{
H^{i,j}(X) \to H^{i,j}(Y) .
}
From here on, we assume that $\Lambda=\Z/m\Z$ with $m$ prime to $p$. 
Even with these coefficients, we cannot expect that the map \eqref{eq.reij} is an
isomorphism for general $i$ and $j$, as the following example shows. 

\begin{ex}[Rosenschon--Srinivas]
In \cite{RS07}, an elliptic curve $E$ over a $p$-adic local field $K$ with finite residue field, such that
$\CH^2(E^3)_\Lambda$ is infinite for certain $m$ ($\Lambda=\Z/m\Z$),
is constructed. By general conjectures, we expect that there exists a model
$X/A$ as above with $X_K\cong E^3$, that there is a surjection $H^{4,2}(X)
\twoheadrightarrow \CH^2(X_K)_\Lambda$ and that $H^{4,2}(Y)$ is finite. Granting this, the
restriction map \eqref{eq.reij} cannot be injective in this case.
\end{ex}

\begin{conj}
Under the condition $j=d$ (higher zero-cycles) or under the condition $i-j=d$ (zero-cycles
with coefficients in Milnor $K$-theory), the map \eqref{eq.reij} is an isomorphism.
\end{conj}

\section*{Acknowledgements}

We thank Spencer Bloch for helpful discussions and the referee for his or her very careful reading.

\bibliographystyle{plain}

\begin{thebibliography}{99}
\bibitem[AK79]{AK79} Altman, A., Kleiman, S.: {\it Bertini theorems for hypersurface sections containing a subscheme}, Comm. Alg. {\bf 7} (1979), no 8, 775--790.
\bibitem[BS98]{BS98} Biswas, J. G., Srinivas, V.: {\it The Chow ring of a singular
    surface}, Proc. Indian Acad. Sci (Math. Sci.), Vol. {\bf 108}, (1998), no 3, 227--249.
\bibitem[CT95]{CT95} Colliot-Th\'el\`ene, J.-L.: {\it
L'arithm\'etique du groupe de Chow des z\'ero-cycles},
J. Th\'eor. Nombres Bordeaux {\bf 7} (1995), no. 1, 51--73. 
\bibitem[SGA4.5]{SGA4.5} Deligne, P.: {\it S\'eminaire de G\'eom\'etrie Alg\'ebrique $4
    \frac{1}{2}$: Cohomologie \'Etale}, Lecture Notes in Mathematics {\bf 569} (1977),
  Springer Verlag.
\bibitem[dJ96]{dJ96} de Jong, A.: {\it
Smoothness, semi-stability and alterations}, 
Publ. Math. IHES {\bf 83} (1996), 51--93. 
\bibitem[De12]{De12} D\'eglise, F.: {\it
Around the Gysin triangle I}, Regulators, 77--116, 
Contemp. Math. {\bf 571},  (2012). 
\bibitem[EW13]{EW13} Esnault, H., Wittenberg, O.: {\it  On the cycle map for zero-cycles over local fields}, with an appendix by Spencer Bloch,
Ann. Sci. ENS {\bf 49} (2016), no. 2, 483--520.
\bibitem[Fo15]{Fo15} Forr\'e, P.: {\it
The kernel of the reciprocity map of varieties over local fields}, 
J. reine angew. Math. {\bf 698} (2015), 55--69. 
\bibitem[FV00]{FV00} Friedlander, E., Voevodsky, V.: {\it Bivariant Cycle Cohomology},  in  Cycles, transfers, and motivic homology theories, 188--238,  Ann. of Math. Stud., {\bf 143} (2000), Princeton Univ. Press, Princeton.
\bibitem[Ful84]{Ful84} Fulton, W.: {\it Intersection Theory},  Ergbenisse der Mathematik und ihrer  Grenzgebiete, 3. Folge, Bd. {\bf 2} (1984), Springer Verlag. 
\bibitem[GLL13]{GLL13} Gabber, O., Liu, Q., Lorenzini, D.: {\it The index of an algebraic variety}, Invent. math. {\bf  192} (2013),  no 3,  567--626.
\bibitem[Gi05]{Gil05} Gillet, H.: {\it  K-Theory and Intersection Theory}, in Handbook of K\nobreakdash-Theory, Vol. 
{\bf 1} (2005), 235--293, Springer Verlag.
\bibitem[HY07]{HY} Hornbostel, J., Yagunov, S.: {\it
Rigidity for Henselian local rings and $\mathbb A^1$-representable theories}, Math. Z.
{\bf 255} (2007), no 2, 437--449. 
\bibitem[ILO]{ILO} Illusie, L., Laszlo, Y., Orgogozo, F.: {\it 
 Travaux de Gabber sur l'uniformisation locale et la cohomologie \'etale des 
sch\'emas quasi-excellents}, Ast\'erisque {\bf 363--364} (2014).
\bibitem[JS12]{JS12} Jannsen, U., Saito, S.: {\it
Bertini theorems and Lefschetz pencils over discrete valuation rings, with applications to higher class field theory},
J. Algebraic Geom. {\bf 21} (2012), no. 4, 683--705. 
\bibitem[Kel13]{Kel13} Kelly, S.: {\it Triangulated categories of motives in positive characteristic}, {\tt http://arxiv.org/abs/1305.5349}, 160 pages.
\bibitem[Ker09]{Ker09} Kerz, M.: {\it The Gersten conjecture for Milnor K-theory}, Invent. math. {\bf  175}  (2009), no. 1, 1--33. 
\bibitem[KS12]{KS12} Kerz, M., Saito, S.: {\it Cohomological Hasse principle and motivic
    cohomology of arithmetic schemes}, Publ.  Math. IHES {\bf 115}
  (2012),  123--183.
\bibitem[Lev98]{Lev98} Levine, M.: {\it Mixed Motives}, Math. Surveys and Monogpraphs {\bf
    57}, AMS (1998). 
\bibitem[LW85]{LW85} Levine, M., Weibel, C.: {\it Zero-cycles and complete intersections on affine surfaces}, 
J. reine angew. Math. {\bf 359} (1985) 106--120.
\bibitem[MVW]{MVW} Mazza, C., Voevodsky, V., Weibel, C.: {\it
Lecture notes on motivic cohomology}, 
 American Mathematical Society (2006).
\bibitem[P08]{Po} Poonen, B.: {\it Smooth hypersurface sections containing
a given subscheme over a finite field},
Math. Res. Lett. {\bf 15} (2008), no. 2, 265--271. 
\bibitem[RS07]{RS07} Rosenschon, A., Srinivas, V.: {\it 
Algebraic cycles on products of elliptic curves over $p$-adic fields},
Math. Ann. {\bf 339} (2007), no. 2, 241--249. 
\bibitem[R96]{R96} Rost, M.: {\it
Chow groups with coefficients}, 
Doc. Math. {\bf 1} (1996), No. 16, 319--393. 
\bibitem[Sc07]{Sch} Schmidt, A.: {\it Singular homology of arithmetic
schemes},
Algebra \& Number Theory~{\bf 1} (2007),  no. 2, 183--222.
\bibitem[SS10]{SS10} Saito, S., Sato, K.: {\it A finiteness theorem for zero-cycles over
    p-adic fields},  with an appendix by Uwe Jannsen,  Ann. of Math. (2) {\bf 172} (2010),
  no. 3, 1593--1639.
\bibitem[Sp12]{Sp12} Spitzweck, M.: {\it A commutative $\mathbb P^1$-spectrum representing
    motivic cohomology over Dedekind domains}, Preprint (2012), arXiv:1207.4078
\bibitem[SV96]{SV96}  Suslin, A., Voevodsky, V.: {\it  Singular homology of abstract
    algebraic varieties}, Invent. math. {\bf 123} (1996), 61--94.
\bibitem[SV02]{SV02} Suslin, A., Voevodsky, V.: {\it Bloch-Kato conjecture and motivic cohomology with finite coefficients},
in
 The arithmetic and geometry of algebraic cycles (Banff, AB, 1998),
 NATO Vol. {\bf 548} (2000), Kluwer Acad. Publ., 117--189.
\bibitem[Voe00]{Voe00b} Voevodsky, V.: {\it Triangulated categories of motives over a
    field}, in  Cycles, transfers, and motivic homology theories, 188--238,  Ann. of Math.
  Stud., {\bf 143} (2000), Princeton Univ. Press, Princeton.
\bibitem[Voe11]{Voe11}   Voevodsky, V.: {\it On motivic cohomology with
    $\Z/\ell$-coefficients}, Ann. of Math. {\bf 174} (2011), 401--438.
\bibitem[W13]{W13} Weibel, C.: {\it
The $K$-book},  Graduate Studies in Mathematics {\bf 145}, AMS (2013).
\end{thebibliography}

\renewcommand\refname{References}

\end{document}